\documentclass[a4paper,10pt]{article}
\usepackage[utf8x]{inputenc}
\usepackage[dvips]{epsfig}
\usepackage{amsmath,amssymb,amsthm,xspace,float}
\usepackage{amsfonts}
\usepackage{color}
\usepackage[dvips]{graphicx}

\def\ie{{\em i.e.}\xspace}

\newcommand{\C}{\ensuremath{\mathcal{C}}\xspace}
\newcommand{\M}{\ensuremath{\mathcal{M}}\xspace}
\newcommand{\B}{\ensuremath{\mathcal{B}}\xspace}
\newcommand{\T}{\ensuremath{\mathcal{T}}\xspace}
\newcommand{\A}{\ensuremath{\mathcal{A}}\xspace}
\newcommand{\V}{\ensuremath{\mathcal{V}}\xspace}
\newcommand{\R}{\ensuremath{\mathcal{R}}\xspace}
\newcommand{\F}{\ensuremath{\mathcal{F}}\xspace}
\newcommand{\G}{\ensuremath{\mathcal{G}}\xspace}
\renewcommand{\H}{\ensuremath{\mathcal{H}}\xspace}
\newcommand{\K}{\ensuremath{\mathcal{K}}\xspace}
\newcommand{\J}{\ensuremath{\mathcal{J}}\xspace}
\newcommand{\X}{\ensuremath{\mathcal{X}}\xspace}

\newcommand{\sym}{\ensuremath{\mathfrak{S}}\xspace}
\newcommand{\sympattern}{\ensuremath{\preccurlyeq_{\sym}}\xspace}
\newcommand{\poly}{\ensuremath{\mathfrak{P}}\xspace}
\newcommand{\polypattern}{\ensuremath{\preccurlyeq_{\poly}}\xspace}
\newcommand{\matr}{\ensuremath{\mathfrak{M}}\xspace}
\newcommand{\matrpattern}{\ensuremath{\preccurlyeq_{\matr}}\xspace}
\newcommand{\quasip}{\ensuremath{\mathfrak{Q}}\xspace}

\newcommand{\Av}{Av}
\newcommand{\Avperm}{Av_{\mathfrak{S}}}
\newcommand{\Avp}{Av_{\mathfrak{P}}}
\newcommand{\Avm}{Av_{\mathfrak{M}}}
\newcommand{\Avquasip}{Av_{\mathfrak{Q}}}

\newtheorem{thm}{Theorem}
\newtheorem{definition}[thm]{Definition}
\newtheorem{remark}[thm]{Remark}
\newtheorem{proposition}[thm]{Proposition}
\newtheorem{example}[thm]{Example}
\newtheorem{corollary}[thm]{Corollary}

\title{Permutation classes and polyomino classes with excluded submatrices}

\author{Daniela Battaglino \thanks{Dipartimento di Ingegneria dell'Informazione e Scienza Matematiche, Via Roma, 56, 53100,
Siena, Italy {\tt battaglino3@unisi.it,  rinaldi@unisi.it}} \and Mathilde Bouvel \thanks{Institut f\"ur Mathematik, Universit\"at Z\"urich, Winterthurerstrasse 190, CH-8057 Z\"urich, Switzerland {\tt mathilde.bouvel@math.uzh.ch}} \and Andrea Frosini \thanks{Dipartimento di Matematica e Informatica, viale Morgagni 67, 50134, Firenze, Italy {\tt andrea.frosini@unifi.it}} \and  Simone Rinaldi $^*$}

\date{}

\begin{document}

\maketitle

\begin{abstract}
This article introduces an analogue of permutation classes in the context of polyominoes. 
For both permutation classes and polyomino classes, we present an original way of characterizing them by avoidance constraints (namely, with excluded submatrices) and we discuss how canonical such a description by submatrix-avoidance can be.
We provide numerous examples of permutation and polyomino classes which may be defined and studied from the submatrix-avoidance point of view, and conclude with various directions for future research on this topic. 
\end{abstract}

\section{Introduction}

The concept of {\em substructure} (or {\em pattern}) within a {\em combinatorial structure} 
is an essential notion in combinatorics, whose study has had many developments in various branches of discrete mathematics. 
Among them, the research on permutation patterns and pattern-avoiding permutations 
-- starting in the seventies with the work of Knuth~\cite{Kn} -- 
has become a very active field. 
Nowadays, the research on permutation patterns is being developed in several directions.
One of them is to define and study analogues of the concept of pattern in permutations in other combinatorial objects
such as set partitions~\cite{Go,Kl,Sa}, words~\cite{Bj,Bu}, trees~\cite{DPTW,R}, and paths~\cite{ferrari}. 
The work presented here goes into this direction, and is specifically interested in patterns in polyominoes. 

Polyominoes\footnote{The term polyomino was introduced by Golomb in 1953 during a talk at the Harvard Mathematics Club, which was published later~\cite{gol} and popularized by Gardner in 1957 in~\cite{gardner}.} are discrete objects among the most studied in combinatorics, especially from an enumerative point of view. 
Indeed, their enumeration (w.r.t. the area or the semi-perimeter) is a difficult problem, and is still open. 
In order to probe further into their study, 
a large part of the research on polyominoes consists in studying restricted families of polyominoes 
-- defined by imposing geometrical constraints, such convexity and directedness, see~\cite{mbm,DV}. 
Our work provides a general framework to define such families of polyominoes, by the avoidance of patterns. 

Both permutations and polyominoes may be represented in a natural way by (restricted) binary matrices, 
and it is on this simple fact that we will build our notion of pattern in a polyomino. 
Doing so, we have also been led to consider families of permutations and polyominoes that are defined by the avoidance of some patterns that are not themselves permutations or polyominoes, but rather less restricted binary matrices. 
In addition to this being an original way to look at restricted families of polyominoes, 
we believe it also brings new ideas to the research on permutation patterns. 

The article is organized as follows. 
In Section~\ref{sec:def_classes}, we recall the definitions of permutation patterns and permutation classes, 
define polyomino classes, and describe some of their basic properties. 
More details about permutations and patterns may be found in~\cite{bona}, 
and we refer the reader to~\cite{mbm} for the basic definitions on polyominoes. 
Section~\ref{sec:excluded_submatrices} introduces submatrix-avoidance in permutations and polyominoes, 
and examines how a permutation or polyomino class may be characterized by the avoidance of submatrices, 
defining several notions of \emph{matrix-bases} of classes. 
Section~\ref{sec:from_one_basis_to_another} investigates how these matrix-bases are related to the usual basis of a permutation class (and its analogue for a polyomino class). 
In Sections~\ref{sec:known_classes_perm} and~\ref{sec:known_classes_poly}, 
we provide several examples of permutation and polyomino classes, both previously studied and new, 
that can be studied with our submatrix-avoidance approach. 
Finally, we open directions for future work in Section~\ref{sec:further_research}. 

\section{Permutation classes and polyomino classes}
\label{sec:def_classes}

\subsection{Permutation patterns and permutation classes}

For any integer $n$, let us denote by $\sym_n$ the set of all permutations of the set $\{1,2, \ldots, n\}$.
The integer $n$ is called the \emph{size} of a permutation in $\sym_n$.
We denote by $\sym = \cup_{n \in \mathbb{N}}\sym_n$ the set of all permutations.
We write permutations in one-line notation $\sigma = \sigma_1 \sigma_2 \ldots \sigma_n$, meaning for instance that $\sigma = 53142$ is the permutation of $\sym_5$ such that
$\sigma(1) =5$, $\sigma(2) =3$, \ldots, $\sigma(5) =2$.

\begin{definition}
Given two permutations $\sigma = \sigma_1 \sigma_2 \ldots \sigma_n$ and $\pi = \pi_1 \pi_2 \ldots \pi_k$,
$\pi$ is a \emph{pattern} of $\sigma$ (denoted $\pi \sympattern \sigma$) when there exist indices $1 \leq i_1 < i_2 < \ldots < i_k \leq n$ such that
the sequence $\sigma_{i_1} \sigma_{i_2} \ldots \sigma_{i_k}$ is order-isomorphic to $\pi$.
Such sequences $\sigma_{i_1} \sigma_{i_2} \ldots \sigma_{i_k}$ are called \emph{occurrences} of $\pi$ in $\sigma$.
\label{def:permutation_pattern}
\end{definition}

A permutation $\pi$ which is a pattern of a permutation $\sigma$ is said to be \emph{contained} or \emph{involved} in $\sigma$.
A permutation $\sigma$ that does not contain $\pi$ is said to \emph{avoid} $\pi$.

The relation $\sympattern$ is a partial order on the set $\sym$ of all permutations.
Moreover, properties of the poset $(\sym,\sympattern)$ have been described in the literature~\cite{poset}
and we recall some of the most well-known here:
$(\sym,\sympattern)$ is a well-founded poset (\ie it does not contain infinite descending chains),
but it is not well-ordered, since it contains infinite \emph{antichains} (\ie infinite sets of pairwise incomparable elements);
moreover, it is a graded poset (the rank function being the size of the permutations).

\begin{definition}
A \emph{permutation class} (sometimes called \emph{pattern class} or \emph{class} for short) is a set of permutations \C that is downward closed for $\sympattern$:
for all $\sigma \in \C$, if $\pi \sympattern \sigma$, then $\pi \in \C$.
\label{def:permutation_class}
\end{definition}

Permutation classes are then just \emph{order ideals} in the poset $(\sym,\sympattern)$. 
In this article, we will also consider order ideals in other posets, and we will call them \emph{classes} as well. 
This choice in the terminology, which could be surprising to a reader more familiar with poset theory, 
is in accordance with the huge literature on permutation classes, where we took the motivations for our work. 

\smallskip

For any set $\B$ of permutations, denoting $\Avperm(\B)$ the set of all permutations that avoid every pattern in $\B$,
we clearly have that $\Avperm(\B)$ is a permutation class.
The converse statement is also true. Namely:

\begin{proposition}
For every permutation class \C, there is a unique antichain $\B$ such that $\C = \Avperm(\B)$.
The set $\B$ consists of all minimal permutations (in the sense of $\sympattern$) that do not belong to \C.
\label{prop:perm_basis}
\end{proposition}

In the usual terminology, $\B$ is called the \emph{basis} of \C.
Here, we shall rather call $\B$ the \emph{permutation-basis} (or \emph{$p$-basis} for short), to distinguish from other kinds of bases that we introduce in Section~\ref{sec:excluded_submatrices}.

Proposition~\ref{prop:perm_basis} is a classical fact in the permutation pattern field; for a proof, see for instance~\cite{bona}.
Notice that because $(\sym,\sympattern)$ contains infinite antichains, the $p$-basis of a permutation class may be infinite.

Actually, Proposition~\ref{prop:perm_basis} does not hold only for permutation classes,
but for downward closed sets in any \emph{well-founded poset}, \emph{i.e.} partially ordered set that do not contain infinite descending chains.

\begin{proposition}
Let $(\mathfrak{X},\preccurlyeq)$ be a well-founded poset.
For any subset $\C$ of $\mathfrak{X}$ that is downward closed for $\preccurlyeq$,
there exists a unique antichain $\B$ of $\mathfrak{X}$ such that
$\C = \Av_{\mathfrak{X}}(\B) = \{ x \in \mathfrak{X} : $ for all $b \in \B, b \preccurlyeq x$ does not hold$\}$.
The set $\B$ consists of all minimal elements of $\mathfrak{X}$ (in the sense of $\preccurlyeq$) that do not belong to \C.
\label{prop:basis_of_downward_closed_subposet}
\end{proposition}

\begin{proof}
Let $\C$ be a subset of $\mathfrak{X}$ that is downward closed for $\preccurlyeq$.
The complement $\mathfrak{X} \setminus \C$ of $\C$ with respect to $\mathfrak{X}$ is upward closed for $\preccurlyeq$.
Let us define $\B$ to be the set of minimal elements of $\mathfrak{X} \setminus \C$:
$\B = \{b \in \mathfrak{X} \setminus \C  : \forall x \in \mathfrak{X} \setminus \C, \text{ if } x \preccurlyeq b \text{ then }x=b\}$.
This is equivalent to characterizing $\B$ as the set of minimal elements of $\mathfrak{X}$ (in the sense of $\preccurlyeq$) that do not belong to \C.
Because $\mathfrak{X}$ is well-founded, we have that $x \in \mathfrak{X} \setminus \C$ if and only if $\exists b \in \B$ such that $b \preccurlyeq x$. 
The contrapositive gives $\C = \Av_{\mathfrak{X}}(\B)$.
In addition, by minimality, the elements of $\B$ are pairwise incomparable, so that $\B$ is indeed an antichain.

To ensure uniqueness, it is enough to notice that for two different antichains $\B$ and $\B'$
the sets $\C = \Av_{\mathfrak{X}}(\B)$ and $\C' = \Av_{\mathfrak{X}}(\B')$ are also different.
\end{proof}

\subsection{Permutation matrices and the submatrix order}
\label{subsec:submatrix}

Permutations are in (obvious) bijection with permutation matrices,
\ie (necessarily square) binary matrices with exactly one entry $1$ in each row and in each column.
To any permutation $\sigma$ of $\sym_n$, we may associate a permutation matrix $M_{\sigma}$ of dimension $n$
by setting $M_{\sigma}(i,j) = 1$ if $i = \sigma(j)$, and $0$ otherwise.
Throughout this work we adopt the convention that rows of matrices are numbered from bottom to top,
so that the $1$ in $M_{\sigma}$ are at the same positions than the dots in the diagram of $\sigma$ -- see an example on Figure~\ref{fig:perm-matrix_diagram}.

\begin{figure}[htd]
\begin{center}
\includegraphics[width=8cm]{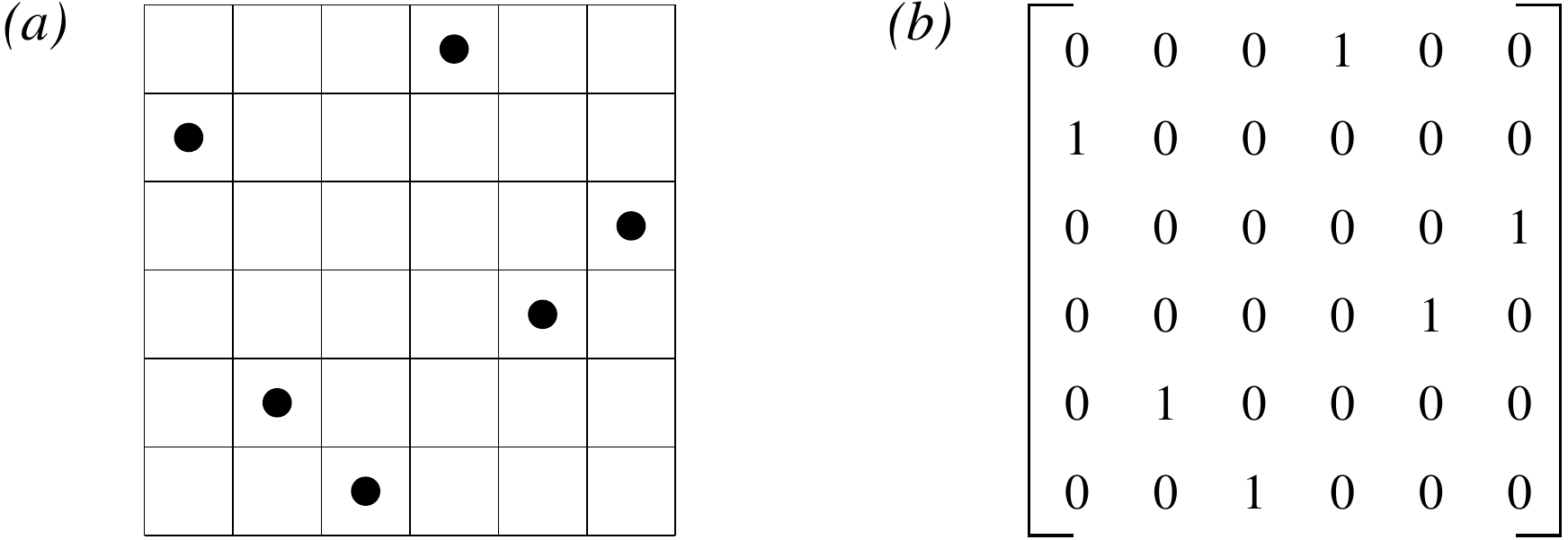}
\caption{$(a)$ Graphical representation (or diagram) of the permutation $\sigma=521634$. $(b)$ The permutation matrix corresponding to $\sigma$.}
\label{fig:perm-matrix_diagram}
\end{center}
\end{figure}

Let \matr be the class of binary matrices (\emph{i.e.} with entries in $\{0,1\}$).
We denote by $\preccurlyeq$ the usual submatrix order on \matr,
\emph{i.e.} $M' \preccurlyeq M$ if $M'$ may be obtained from $M$ by deleting any collection of rows and/or columns.
Of course, whenever $\pi \sympattern \sigma$, we have that $M_{\pi}$ is a submatrix of $M_{\sigma}$.

It follows from identifying permutations with the corresponding permutation matrices
that we may rephrase the definition of permutation classes as follows:
A set \C of permutations is a class if and only if, for every $\sigma \in \C$, every submatrix of $M_{\sigma}$ which is a permutation is in \C.
This does not say much by itself,
but it allows to define analogues of permutation classes for other combinatorial objects that are naturally represented by matrices, like polyominoes.

\subsection{Polyominoes and polyomino classes}\label{sec:papc}

\begin{definition}
A polyomino is a finite union of cells (\ie unit squares in the plane $\mathbb{Z} \times\mathbb{Z}$) that is connected and has no cut point
(\emph{i.e.} the set of cells has to be connected according to the edge adjacency).
Polyominoes are defined up to translation.
\label{def:polyomino}
\end{definition}

A polyomino $P$ may be represented by a binary matrix $M$ whose dimensions are those of the minimal bounding rectangle of $P$:
drawing $P$ in the positive quarter plane, in the unique way that $P$ has contacts with both axes,
an entry $(i,j)$ of $M$ is equal to $1$ if the unit square $[j-1,j]\times[i-1,i]$ of $\mathbb{Z} \times\mathbb{Z}$
is a cell of $P$, $0$ otherwise (see Figure~\ref{polMatr}).
Notice that, according to this definition, in a matrix representing a polyomino the first (resp. the last) row (resp. column) should contain at least a $1$.

\begin{figure}[htd]
\begin{center}
\includegraphics[width=8cm]{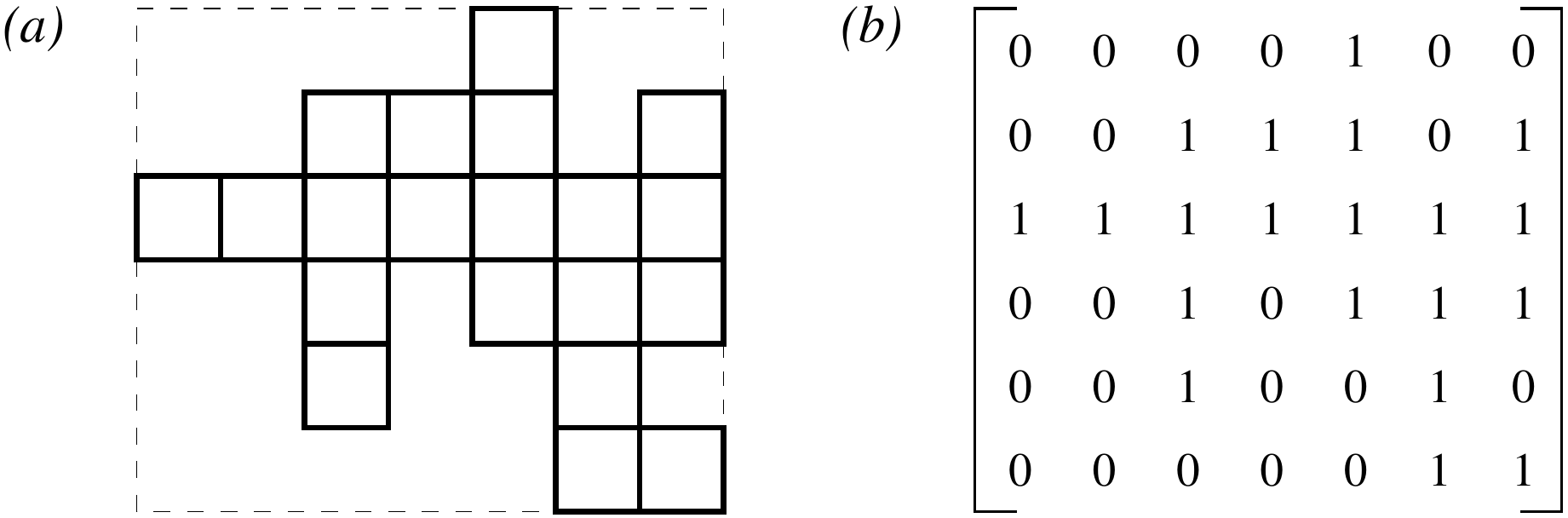}
\caption{A polyomino and its representation as a binary matrix.}
\label{polMatr}
\end{center}
\end{figure}

Let us denote by \poly the set of polyominoes, viewed as binary matrices as explained above.
We can consider the restriction of the submatrix order $\preccurlyeq$ on \poly.
This defines the poset $(\poly,\polypattern)$ and the pattern order between polyominoes:
a polyomino $P$ is a \emph{pattern} of a polyomino $Q$ (which we denote $P \polypattern Q$)
when the binary matrix representing $P$ is a submatrix of that representing $Q$.

We point out that the order $\polypattern$ has already been studied in~\cite{CR} under the name of {\em subpicture order}.
The main point of focus of~\cite{CR} is the family of {\em $L$-convex polyominoes} defined by the same authors in~\cite{CR1}.
But~\cite{CR} also proves that $\polypattern$ is not a partial well-order, since $(\poly,\polypattern)$ contains infinite antichains.
Remark also that $(\poly,\polypattern)$ is a graded poset (the rank function being the semi-perimeter of the bounding box of the polyominoes). 
This implies in particular that $(\poly,\polypattern)$ is well-founded.
Notice that these properties are shared with the poset $(\sym,\sympattern)$ of permutations. 

A natural analogue of permutation classes for polyominoes is as follows:

\begin{definition}
A \emph{polyomino class} is a set of polyominoes \C that is downward closed for $\polypattern$:
for all polyominoes $P$ and $Q$, if $P \in \C$ and $Q \polypattern P$, then $Q \in \C$.
\label{def:polyomino_class}
\end{definition}

The reader can exercise in finding simple examples of polyomino classes, like 
the family of polyominoes having at most $k$ columns, for any fixed $k$, or
the family of polyominoes having a rectangular shape.
Some of the most famous families of polyominoes are indeed polyomino classes,
like the {\em convex polyominoes} and the $L$-convex polyominoes.
This will be investigated in more details in Section~\ref{sec:known_classes_poly}.
However, there are also well-known families of polyominoes which are not polyomino classes, like:
the family of polyominoes having a square shape, 
or the family of polyominoes {\em with no holes} (\ie polyominoes whose boundary is a simple path).
Figure~\ref{fig:buco} shows that a polyomino in this class may contain a polyomino with a hole.

\begin{figure}[ht]
\begin{center}
\includegraphics[width=6cm]{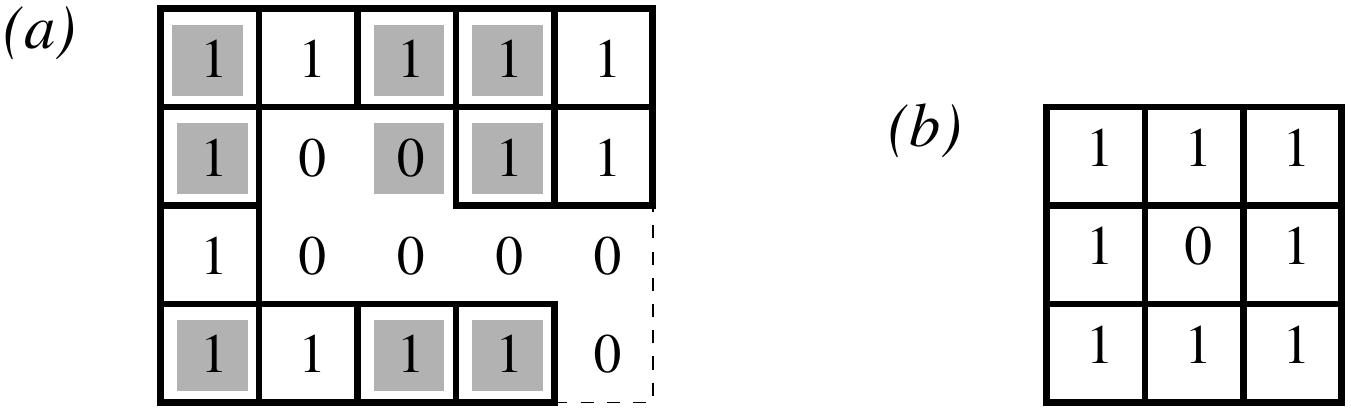}
\caption{$(a)$ A polyomino $P$ with no holes; $(b)$ A polyomino $P' \polypattern P$ containing a hole.}
\label{fig:buco}
\end{center}
\end{figure}

Similarly to the case of permutations, for any set $\B$ of polyominoes,
let us denote by $\Avp(\B)$ the set of all polyominoes that do not contain any element of $\B$ as a pattern.
Every such set $\Avp(\B)$ of polyominoes defined by pattern avoidance is a polyomino class.
Conversely, like for permutation classes, every polyomino class may be characterized in this way.

\begin{proposition}
For every polyomino class \C, there is a unique antichain \B of polyominoes such that  $\C=\Avp(\B)$.
The set \B consists of all minimal polyominoes (in the sense of \polypattern) that do not belong to \C.
\label{prop:polyomino_basis}
\end{proposition}

\begin{proof}
It follows immediately from Proposition~\ref{prop:basis_of_downward_closed_subposet} and the fact that $(\poly,\polypattern)$ is a well-founded poset.
\end{proof}

As in the case of permutations we call \B the \emph{polyomino-basis} (also abbreviated as \emph{$p$-basis}), to distinguish from other kinds of bases introduced in Section~\ref{sec:excluded_submatrices}.

\begin{example}\label{ex:W1}
The set ${\mathcal{V}}$ of vertical bars (\emph{i.e.}, polyominoes with only one column) is a polyomino class 
whose $p$-basis is $\left\{{\footnotesize\left[\begin{array}{cc}
          1 & 1
         \end{array}
 \right]}\right\}$.

Let ${\mathcal{W}}\xspace$ be the set of polyominoes made of at most two columns. 
It is a polyomino class and its $p$-basis is shown in Figure~\ref{fig:W1}.
\end{example}

\begin{figure}[ht]
\begin{center}
\includegraphics[width=6cm]{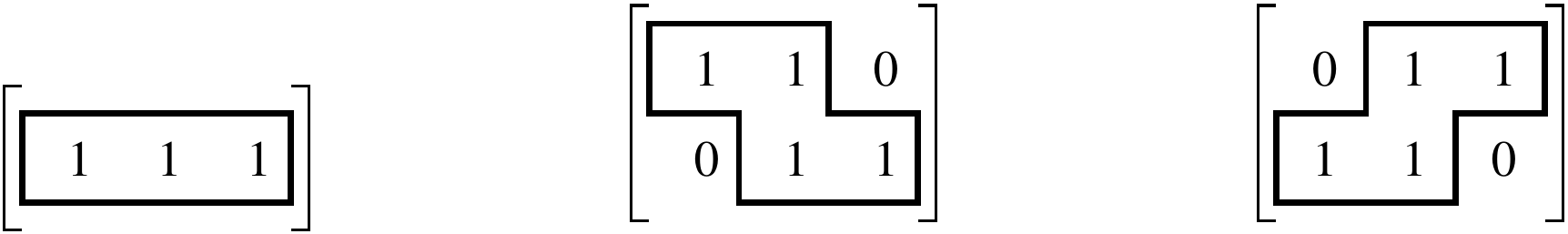}
\caption{The $p$-basis of the class ${\mathcal{W}}\xspace$.}
\label{fig:W1}
\end{center}
\end{figure}

Recall that $(\poly,\polypattern)$ contains infinite antichains~\cite{CR}, so there exist polyomino classes with infinite $p$-basis. 
We will show an example of a polyomino class with an infinite $p$-basis in Proposition~\ref{prop:infinite_basis}.
However, we are not aware of \emph{natural} polyomino classes whose $p$-basis is infinite.

\section{Characterizing classes with excluded submatrices}
\label{sec:excluded_submatrices}

\subsection{Submatrices as excluded patterns in permutations and polyominoes}
\label{subsec:matrix_patterns}

Obviously, not all submatrices of permutation matrices are themselves permutation matrices. 
More precisely:

\begin{remark}
The submatrices of permutation matrices
are exactly those that contain at most one $1$ in each row and each column.
We will call such matrices \emph{quasi-permutation matrices} in the rest of this paper. 
Denoting $\quasip$ the set of quasi-permutation matrices, $(\quasip, \matrpattern)$ 
is a subposet of $(\matr, \matrpattern)$, and is of course also well-founded. 
\label{rem:submatrices_of_perms}
\end{remark}

Quasi-permutation matrices are equivalently described as \emph{rook placements} (of rectangular shape). 
Rook placements avoiding rook placements have recently been studied from an enumerative point of view, 
as well as rook placements avoiding permutations~\cite{DaPu}. 
However, to our knowledge, permutations that avoid rook placements have never been investigated, 
and this is also one way of viewing our work on permutation classes defined by the avoidance of some quasi-permutation matrices 
(see in particular Section~\ref{sec:known_classes_perm}). 

\medskip

Similarly to Remark~\ref{rem:submatrices_of_perms}, not all submatrices of polyominoes are themselves polyominoes,
but the situation is very different from that of permutations:

\begin{remark}
Every binary matrix is a submatrix of some polyomino.
\label{rem:submatrices_of_pol}
\end{remark}

Indeed, for every binary matrix $M$,
it is always possible to add rows and columns of $1$ to $M$
in such a way that all $1$ entries of the resulting matrix are connected. 

\medskip

From Remarks~\ref{rem:submatrices_of_perms} and~\ref{rem:submatrices_of_pol},
it makes sense to examine sets of permutations (resp. polyominoes)
that avoid submatrices that are not themselves permutations (resp. polyominoes) 
but quasi-permutation matrices\footnote{\label{note1}We may as well consider sets $\M$ containing arbitrary binary matrices.
But excluding a matrix $M$ which is not a quasi-permutation matrix
is not actually introducing any restriction, since no permutation contains $M$ as a submatrix.} (resp. binary matrices).

\begin{definition}
For any set $\M$ of quasi-permutation matrices (resp. of binary matrices),
let us denote by $\Avperm(\M)$ (resp. $\Avp(\M)$) the set of all permutations (resp. polyominoes)
that do not contain any submatrix in $\M$.
\label{def:submatrix_avoidance}
\end{definition}

Figure~\ref{motivo} illustrates Definition~\ref{def:submatrix_avoidance} in the polyomino case.

\begin{figure}[ht]
\begin{center}
\includegraphics[width=11cm]{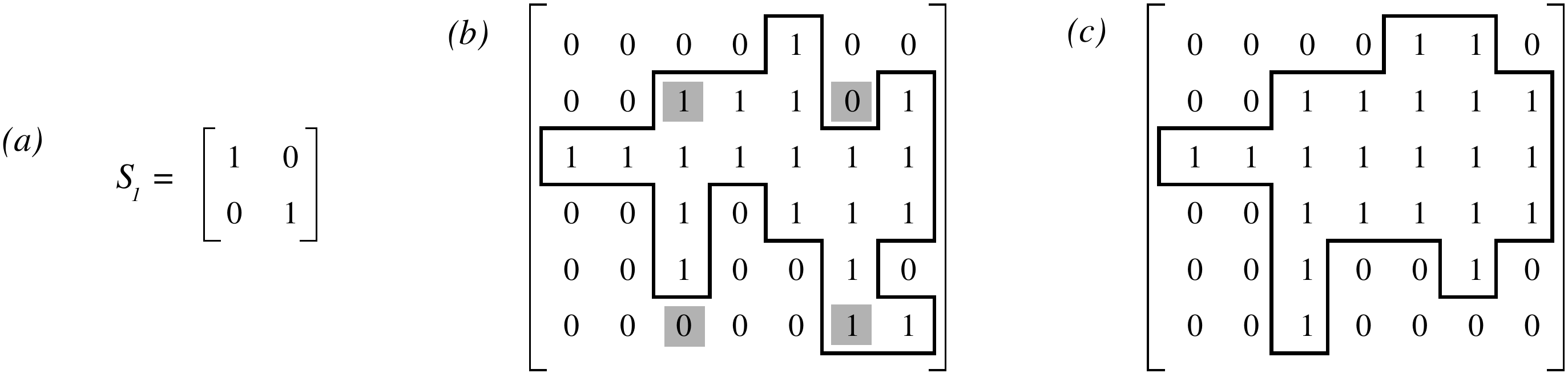}
\caption{$(a)$ a matrix $S_1$; $(b)$ a polyomino that contains $S_1$ as a submatrix, hence does not belong to $\Avp(S_1)$;
$(c)$ a polyomino that does not contain $S_1$, \emph{i.e.} that belongs to $\Avp(S_1)$.}
\label{motivo}
\end{center}
\end{figure}

The following facts, although immediate to prove, will be useful in our work:

\begin{remark}
When $\M$ contains only permutations (resp. polyominoes),
these definitions of $\Avperm(\M)$ and $\Avp(\M)$ coincide with the ones given in Section~\ref{sec:def_classes}.
\end{remark}

\begin{remark}
Denoting $\Avm(\M)$ the set of binary matrices without any submatrix in $\M$, we have
$
\Avperm(\M)=\Avm(\M)\cap\sym \text{ and } \Avp(\M)=\Avm(\M)\cap\poly \text{.}
$

In addition\footnote{Like in footnote \ref{note1}, 
we may assume that $\M$ contains only quasi-permutation matrices here.}, 
denoting $\Avquasip(\M)$ the set of quasi-permutation matrices without any submatrix in $\M$, 
we also have $
\Avperm(\M)=\Avquasip(\M)\cap\sym\text{.}
$
\label{rem:intersection_with_permutations_or_polyominoes}
\end{remark}

\begin{remark}
Sets of the form $\Avperm(\M)$  are downward closed for $\sympattern$, \ie are permutation classes.
Similarly, sets $\Avp(\M)$ are polyomino classes.
\label{rem:submatrix_avoidance_implies_class}
\end{remark}

We believe it is quite natural to characterize some permutation or polyomino classes
by avoidance of submatrices, and will provide several examples in Sections~\ref{sec:known_classes_perm} and~\ref{sec:known_classes_poly}.
In the present section, we investigate further
the description of permutation and polyomino classes by avoidance of matrices,
and in particular how canonical and concise such a description can be.

\subsection{Matrix bases of permutation and polyomino classes}

From Propositions~\ref{prop:perm_basis} and~\ref{prop:polyomino_basis}, 
every permutation (resp. polyomino) class \C is characterized 
by the avoidance of a set of permutations (resp. polyominoes) which is uniquely determined.
With Proposition~\ref{prop:basis_of_downward_closed_subposet},
we may similarly associate with every permutation (resp. polyomino) class $\C$
a canonical set $\M$ of matrices such that $\C = \Avperm(\M)$ (resp. $\C = \Avp(\M)$)
(see Definition~\ref{def:canonical_m-basis}).
However, we shall see that for some permutation (resp. polyomino) classes $\C$,
there exist several antichains $\M'$ such that $\C = \Avperm(\M')$ (resp. $\C = \Avp(\M')$).

\begin{definition}
Let \C be a permutation (resp. polyomino) class.
Denote by $\C^+$ be the set of matrices that appear as a submatrix of some element of \C, \emph{i.e.}
\[
\C^+ = \{ M \in \matr : \exists P \in \C, \text{ such that }M \preccurlyeq P\} \text{.}
\]
Notice that $\C^+$ contains only quasi-permutation (resp. binary) matrices. 

Denote by \M the set of all minimal quasi-permutation (resp. binary) matrices in the sense of $\preccurlyeq$ that do not belong to $\C^+$.
\M is called the \emph{canonical matrix-basis} (or \emph{canonical $m$-basis} for short) of \C.
\label{def:canonical_m-basis}
\end{definition}

Of course, the canonical $m$-basis of a class $\C$ is uniquely defined, and is always an antichain for $\preccurlyeq$. 
Moreover, Proposition~\ref{prop:canonical_m-basis} shows that it indeed provides a description of \C by avoidance of submatrices.

\begin{proposition}
Let \C be a permutation (resp. polyomino) class,
and denote by $\M$ its canonical $m$-basis.
We have $\C = \Avperm(\M)$ (resp. $\C = \Avp(\M)$).
\label{prop:canonical_m-basis}
\end{proposition}

\begin{proof}
In the case of a permutation class $\C$, we can work in the poset $(\quasip, \preccurlyeq)$ 
of quasi-permutation matrices. Proposition~\ref{prop:basis_of_downward_closed_subposet}
ensures that $\C^+ = \Avquasip(\M)$. 
For a polyomino class, working in the poset of binary matrices $(\matr, \preccurlyeq)$, Proposition~\ref{prop:basis_of_downward_closed_subposet}
ensures that $\C^+ = \Avm(\M)$.
And since $\C = \C^+ \cap \sym$ (resp. $\C = \C^+ \cap \poly$),
Remark~\ref{rem:intersection_with_permutations_or_polyominoes} yields the conclusion. 
\end{proof}

\begin{example}[Canonical $m$-basis]
\label{ex:canonical_m-basis}
For the (trivial) permutation class $\T=\{1,12,21\}$, we have
$$\T^{+} = \Big\{ 
{\footnotesize 
\left[ \begin{array}{c}
0 \end{array} \right],
\left[ \begin{array}{c}
1 \end{array} \right],
\left[ \begin{array}{cc}
1 & 0  \end{array} \right],
\left[ \begin{array}{cc}
0 & 1 \end{array} \right],
\left[ \begin{array}{c}
1 \\
0 \end{array} \right],
\left[ \begin{array}{c}
0 \\
1 \end{array} \right],
\left[ \begin{array}{cc}
1 & 0 \\
0 & 1 \end{array} \right],
\left[ \begin{array}{cc}
0 & 1 \\
1 & 0 \end{array} \right]
}
\Big\}$$ 
and the canonical $m$-basis of $\T$ is
$\left\{
{\footnotesize 
  \left[\begin{array}{cc}
          0 & 0
         \end{array}
  \right],
  \left[\begin{array}{c}
          0\\
          0
         \end{array}
  \right]
}  
\right\}
$.

Let $\A$ be the permutation class $\Avperm(321,231,312)$. The canonical $m$-basis of $\A$ is
$
\{Q_1,Q_2\} \text{ with } Q_1 = {\footnotesize \left[\begin{array}{cc}
          1 & 0\\
          0 & 0\\
          0 & 1
         \end{array}
 \right]} \text{ and } Q_2 = {\footnotesize \left[\begin{array}{ccc}
          1 & 0 & 0\\
          0 & 0 & 1
         \end{array}
 \right]}\text{.}
$

The canonical $m$-basis of the class $\V$ of vertical bars, considered in Example~\ref{ex:W1}, is
$
\{{\footnotesize \left[\begin{array}{c}
          0
         \end{array}
 \right]}, {\footnotesize \left[\begin{array}{cc}
          1 & 1
         \end{array}
 \right]}\}
$.

Let $\R$ be the polyomino class consisting of polyominoes of rectangular shape. 
The canonical $m$-basis of $\R$ consists only of the matrix ${\footnotesize \left[\begin{array}{c}
         0
         \end{array}
  \right]}$.

The canonical $m$-basis of the class ${\mathcal{W}}\xspace$ of polyominoes with at most two columns, considered in Example \ref{ex:W1}, is \\
$\big\{ {\footnotesize 
  \left[\begin{array}{ccc}
          1 & 1 & 1 
         \end{array}
  \right],
  \left[\begin{array}{cc}
          0 & 0        
         \end{array}
  \right],} $ $ {\footnotesize 
  \left[\begin{array}{ccc}
          1 & 0 & 1        
         \end{array}
  \right],
  \left[\begin{array}{ccc}
          1 & 1 & 0       
         \end{array}
  \right],
  \left[\begin{array}{ccc}
          0 & 1 & 1       
         \end{array}
  \right] }
\big\}
\text{.}
$
\end{example}

There is one important difference between $p$-basis and canonical $m$-basis.
Every antichain of permutations (resp. polyominoes) is the $p$-basis of a class.
On the contrary, every antichain $\M$ of quasi-permutation (resp. binary) matrices describes a permutation (resp. polyomino) class $\Avperm(\M)$ (resp. $\Avp(\M)$),
but not every such antichain is the canonical $m$-basis of the corresponding permutation (resp. polyomino) class --
see Example~\ref{ex:m-basis} below.
Imposing the avoidance of matrices taken in an antichain being however a natural way of describing permutation and polyomino classes,
let us define the following weaker notion of basis.

\begin{definition}
Let \C be a permutation (resp. polyomino) class.
Every antichain $\M$ of matrices such that $\C=\Avperm(\M)$ (resp. $\Avp(\M)$)
is called a \emph{matrix-basis} (or \emph{$m$-basis}) of $\C$.
\label{def:m-basis}
\end{definition}

Example~\ref{ex:m-basis} 
show several examples of $m$-bases of permutation and polyomino classes
which are different from the canonical $m$-basis (see Example~\ref{ex:canonical_m-basis}).

\begin{example}[$m$-bases]
\label{ex:m-basis}
Consider the set $\M$ consisting of the following four matrices:
$M_1 = {\footnotesize \left[ \begin{array}{cc}
1 & 0 \\
0 & 0 \end{array} \right]}, \, M_2 = {\footnotesize \left[ \begin{array}{cc}
0 & 1 \\
0 & 0 \end{array} \right]}, \,  M_3 = {\footnotesize \left[ \begin{array}{cc}
0 & 0 \\
1 & 0 \end{array} \right]}, \, M_4 = {\footnotesize \left[ \begin{array}{cc}
0 & 0 \\
0 & 1 \end{array} \right]}$.
We may check that every permutation of size $3$ contains a submatrix $M \in \M$,
and that it actually contains each of these four $M_i$.
Moreover, $\M$ is an antichain, and so is obviously each set $\{M_i\}$.
Therefore, $\T = \Avperm(\M) =\Avperm(M_i)$, for each $1\leq i \leq 4$,
even though these antichains characterizing $\T$ are not the canonical $m$-basis of $\T$.

It is readily checked that 
$\A= \Avperm(Q_1)=\Avperm(Q_2)$ even though the canonical $m$-basis of
$\A$ is $\{Q_1,Q_2\}$.

The canonical $m$-basis of the class \V of vertical bars is $
\{{\footnotesize\left[\begin{array}{c}
          0
         \end{array}
 \right],\left[\begin{array}{cc}
          1 & 1
         \end{array}
 \right]}\}
$. But we also have $\Avp\left({\footnotesize\left[\begin{array}{cc}
          1 & 1
         \end{array}
 \right]}\right) = \V$.

Consider the sets
$\M_1=\Big\{ {\footnotesize \left[
  \begin{array}{cc}
  1 & 0\end{array}
\right],
\left[
  \begin{array}{cc}
  0 & 1\end{array}
\right],
\left[
  \begin{array}{cc}
  0 & 0\end{array}
\right],
\left[
  \begin{array}{c}
  0\\
  0\end{array}
\right],}$ ${\footnotesize
\left[
  \begin{array}{c}
  1\\
  0\end{array}
\right],
\left[
  \begin{array}{c}
  0\\
  1\end{array}
\right]}\Big\}$
and
$\M_2=\left\{{\footnotesize\left[
  \begin{array}{cc}
  1 & 0\end{array}
\right],
\left[
  \begin{array}{cc}
  0 & 1\end{array}
\right],
\left[
  \begin{array}{c}
  1\\
  0\end{array}
\right],
\left[
  \begin{array}{c}
  0\\
  1\end{array}
\right]}\right\}$. 
We may easily check that $\M_1$ and $\M_2$ are antichains,
and that their avoidance characterizes the rectangular polyominoes:
$\R=\Avp(\M_1)=\Avp(\M_2)$.

Similarly, the sets 
${\cal M}_3=\left\{{\footnotesize 
  \left[ \begin{array}{ccc}
          1 & 1 & 1 
         \end{array}
  \right],
  \left[\begin{array}{ccc}
          1 & 1 & 0       
         \end{array}
  \right]}
\right\}$ and ${\cal M}_4=\{{\footnotesize
  \left[\begin{array}{ccc}
          1 & 1 & 1 
         \end{array}
  \right],}$ ${\footnotesize
  \left[\begin{array}{ccc}
          0 & 1 & 1       
         \end{array}
  \right]}
\}
$ 
are $m$-bases of the class ${\mathcal{W}}\xspace$ of polyominoes with at most two columns. 
\end{example}

Example~\ref{ex:m-basis} shows in addition that the canonical $m$-basis is not always the most concise way
of describing a permutation (or polyomino) class 
by avoidance of submatrices. This motivates the following definition:

\begin{definition}
Let \C be a permutation (resp. polyomino) class.
A \emph{minimal $m$-basis} of $\C$ is an $m$-basis of $\C$ satisfying the following additional conditions:
\begin{itemize}
 \item[$(1.)$] \M is a minimal subset subject to $\C=\Avperm(\M)$ (resp. $\Avp(\M)$), \\
 \emph{i.e.} for every strict subset $\M'$ of $\M$, $\C \neq \Avperm(\M')$ (resp. $\Avp(\M')$);
 \item[$(2.)$] for every submatrix $M'$ of some matrix $M \in \M$, we have 
\begin{itemize}
  \item[$i.$] $M'=M$ or
  \item[$ii.$] with $\M'=\big( \M\setminus\{M\} \big) \cup\{M'\}$, $\C \neq \Avperm(\M')$ (resp. $\Avp(\M')$).
 \end{itemize}
\end{itemize}
\label{def:minimal_m-basis}
\end{definition}

Condition~$(1.)$ ensures minimality in the sense of inclusion,
while Condition~$(2.)$ ensures that it is not possible to replace a matrix of the minimal $m$-basis
by another one of smaller dimensions.
For future reference, let us notice that with the notations of Definition~\ref{def:minimal_m-basis},
the statement $\C \neq \Avperm(\M')$ (resp. $\Avp(\M')$) in Condition~$(2.)ii.$ is equivalent to
$\C \nsubseteq \Avperm(\M')$ (resp. $\Avp(\M')$), since the other inclusion always holds.

To illustrate the relevance of Condition~$(2.)$, consider for instance
the $m$-basis $\{M_1\}$ of $\T$, with $M_1 = {\footnotesize \left[ \begin{array}{cc}
1 & 0 \\
0 & 0 \end{array} \right]}$ (see Example~\ref{ex:m-basis}).
Of course it is minimal in the sense of inclusion, however noticing that
$\T = \Avperm\left({\footnotesize\left[ \begin{array}{cc} 0 & 0 \end{array} \right]}\right) = \Avperm\left({\footnotesize\left[ \begin{array}{c}
 0 \\
 0 \end{array} \right]}\right)$, with these excluded submatrices being submatrices of $M_1$,
it makes sense \emph{not} to consider $\{M_1\}$ as a \emph{minimal} $m$-basis.
This is exactly the point of Condition~$(2.)$.
Actually, $\left\{{\footnotesize\left[ \begin{array}{cc} 0 & 0 \end{array} \right]}\right\}$ and
$\left\{{\footnotesize\left[ \begin{array}{c}
 0 \\
 0 \end{array} \right]} \right\}$ both satisfy Conditions~$(1.)$ and~$(2.)$, \emph{i.e.}
are minimal $m$-basis of $\T$.

This also illustrates the somewhat undesirable property that a class may have several minimal $m$-bases.
This is not only true for the trivial class $\T$. 

\begin{example}\label{ex:minimal_m-basis}
For the permutation class $\A$, 
the $m$-bases $\{Q_1\}$ and $\{Q_2\}$ of $\A$ (see Example~\ref{ex:canonical_m-basis})
are minimal $m$-bases of $\A$.

For the polyomino class ${\mathcal{W}}\xspace$, 
the two $m$-bases ${\cal M}_3$ and ${\cal M}_4$ of Example~\ref{ex:m-basis} 
are minimal $m$-bases of ${\mathcal{W}}\xspace$. 
We can observe that in this case each polyomino containing a pattern of ${\cal M}_3$ must contain a pattern of ${\cal M}_4$ and conversely. 
\end{example}

However, the minimal $m$-bases of a class are relatively constrained:

\begin{proposition}
Let \C be a permutation (resp. polyomino) class and let $\M$ be its canonical $m$-basis.
The minimal $m$-bases of $\C$ are the subsets $\B$ of $\M$ that are minimal (for inclusion) under the condition $\C=\Avperm(\B)$ (resp. $\Avp(\B)$).
\label{prop:minimal_m-basis}
\end{proposition}

\begin{proof}
For simplicity of notations, let us forget the indices and write $\Av(\B)$ instead of $\Avperm(\B)$ (resp. $\Avp(\B)$).

Consider a subset $\B$ of $\M$ that is minimal for inclusion under the condition $\C=\Av(\B)$,
and let us prove that $\B$ is a minimal $m$-basis of $\C$.
$\B$ is clearly an $m$-basis of $\C$ satisfying Condition~$(1.)$.
Assume that $\B$ does not satisfy Condition~$(2.)$: there is some $M \in \B$ and some proper submatrix $M'$ of $M$,
such that $\C=\Av(\B')$ for $\B'=(\B\setminus\{M\})\cup\{M'\}$.
By definition of the canonical $m$-basis, $M' \in \C^+$ (or $M$ would not be minimal for $\preccurlyeq$),
so there exists a permutation (resp. polyomino) $P \in \C$ such that $M' \preccurlyeq P$.
But then $P \notin \Av(\B') = \C$ bringing the contradiction which ensures that $\B$ satisfies Condition~$(2.)$.

Conversely, consider a minimal $m$-basis $\B$ of $\C$ and a matrix $M \in \B$,
and let us prove that $M$ belongs to $\M$.
Because of Condition~$(1.)$, this is enough to conclude the proof.
First, notice that $M \notin \C^+$.
Indeed, otherwise there would exist a permutation (resp. polyomino) $P \in \C$ such that $M \preccurlyeq P$,
and we would also have $P \notin \Av(\B) = \C$, a contradiction.
By definition, $\C^+ = \Avquasip(\M)$ (resp. $\C^+ = \Avm(\M)$), so there exists $M' \in \M$ such that $M' \preccurlyeq M$.
Since $\B$ is a minimal $m$-basis we either have $M=M'$, which proves that $M \in \M$,
or we have $\Av(\B') \varsubsetneq \C$ for $\B'=(\B\setminus\{M\})\cup\{M'\}$,
in which case we derive a contradiction as follows.
If $\Av(\B') \varsubsetneq \C$, then there is some permutation (resp. polyomino) $P \in \C$
which has a submatrix in $\B'$. It cannot be some submatrix in $\B\setminus\{M\}$, because $\C = \Av(\B)$.
So $M' \preccurlyeq P$, which is a contradiction to $P \in \C = \Av(\M)$.
\end{proof}

\begin{example}\label{ex:all_minimal_m-basis}
Proposition~\ref{prop:minimal_m-basis} 
(together with Examples~\ref{ex:canonical_m-basis} and~\ref{ex:m-basis})
allows to compute all minimal $m$-bases for our running examples. 

Both permutation classes $\T$ and $\A$ each have exactly two minimal $m$-bases,
namely $\left\{{\footnotesize \left[ \begin{array}{cc} 0 & 0 \end{array} \right]}\right\}$ and
$\left\{{\footnotesize\left[ \begin{array}{c}
 0 \\
 0 \end{array} \right]} \right\}$,
and $\{Q_1\}$ and $\{Q_2\}$ respectively.

The polyomino classes $\V$ and $\R$ have unique minimal $m$-bases, respectively 
$\left\{{\footnotesize\left[\begin{array}{cc}
          1 & 1
         \end{array}
 \right]}\right\}$ and $\left\{{\footnotesize\left[\begin{array}{c}
         0
         \end{array}
  \right]}\right\}$. 
But the polyomino class ${\mathcal{W}}\xspace$ has exactly two minimal $m$-bases,
namely ${\cal M}_3$ and ${\cal M}_4$.
\end{example}

A natural problem is then to characterize the permutation (resp. polyomino) classes
which have a unique minimal $m$-basis. 
Propositions~\ref{prop:when_p-basis_is_minimal} and~\ref{prop:uniqueness_minimal_m-basis} give partial answers to this problem. 

\begin{proposition}   
\label{prop:when_p-basis_is_minimal}
If the $p$-basis of a permutation (resp. polyomino) class \C is also a minimal $m$-basis of \C, 
then it is the unique minimal $m$-basis of \C.
\end{proposition}

\begin{proof}
Consider a permutation (resp. polyomino) class \C, denote by $\B$ its $p$-basis, 
and assume that $\B$ is a minimal $m$-basis of \C. 
Let us also denote $\M$ the canonical $m$-basis of \C. 
From Proposition~\ref{prop:minimal_m-basis}, $\B$ is a subset of $\M$, 
and we see that Proposition~\ref{prop:when_p-basis_is_minimal} will follow if we prove that: 
for all subsets $\X$ of $\M$, \\
\hspace*{0.3cm} $(i)$ either $\X=\B$,\\
\hspace*{0.3cm} $(ii)$ or $\C\neq\Avperm(\X)$ (resp. $\Avp(\X)$),\\
\hspace*{0.3cm} $(iii)$ or $\X$ contains a strict subset $\X'$ such that $\C=\Avperm(\X')$ (resp. $\Avp(\X')$).\\
So, let $\X$ be a subset of $\M$, different from $\B$. Assume that $\C=\Avperm(\X)$ (resp. $\Avp(\X)$), 
and let us prove that $(iii)$ holds. 

First we claim that $\B \subseteq \X$. 
Indeed, assuming the contrary, there would exists $b \in \B$ such that $b \notin \X$. 
$\B$ being the $p$-basis of \C, we deduce that $b$ is a permutation (resp. polyomino) such that $b \notin \C$.
Because $\C=\Avperm(\X)$ (resp. $\Avp(\X)$), we know that $b$ contains a pattern in $\X$. 
And since $b \notin \X$, it has to be a strict pattern, that we denote $b'$. 
The claim is then proved, since $b \in \B \subseteq \M$ and $b'\in \X \subseteq \M$ are two distinct elements of $\M$, such that $b' \preccurlyeq b$, 
contradicting that $\M$ is an antichain. 

Knowing that $\B \subseteq \X$ holds, we conclude the proof as follows. 
Because $\X \neq \B$, $\B$ is a strict subset of $\X$. 
By hypothesis  $\C=\Avperm(\B)$ (resp. $\Avp(\B)$), so that $\X$ contains a strict subset $\X'=\B$ such that $\C=\Avperm(\X')$ (resp. $\Avp(\X')$).
\end{proof}

Proposition~\ref{prop:when_p-basis_is_minimal} applies for instance to the class \V of vertical bars (see Examples \ref{ex:W1} and \ref{ex:all_minimal_m-basis}) 
or to the case of parallelogram polyominoes (in Section~\ref{sec:known_classes_poly}).

\begin{proposition}   
\label{prop:uniqueness_minimal_m-basis}
Let \C be a permutation (resp. polyomino) class and \M be its canonical $m$-basis. 
Assuming that for all $M\in$ \M, 
there exists a permutation (resp. polyomino) $P$ such that $M \preccurlyeq P$ and
for all $M'\in$ \M (with $M'\neq M$) $M'{\not\preccurlyeq} P$, 
then \M is a minimal $m$-basis. Consequently \C has a unique minimal $m$-basis.
\end{proposition}

\begin{proof}
Considering a proper subset $\M'$ of \M, from Proposition \ref{prop:minimal_m-basis} it is enough to prove that $\Avperm(\M')\neq\Avperm(\M)$ (resp. $\Avp(\M')\neq\Avp(\M)$).
Since $\M'\varsubsetneq \M$ there exists $M''\in \M$ such that $M''\notin\M'$. 
By assumption, there is a permutation (resp. polyomino) $P$ such that $M''\preccurlyeq P$ but $M'{\not\preccurlyeq} P$ for all $M' \in \M, M'\neq M''$. Hence $\Avperm(\M')\neq\Avperm(\M)$ (resp. $\Avp(\M')\neq\Avp(\M)$).
\end{proof}

\section{Relations between the $p$-basis and the $m$-bases}
\label{sec:from_one_basis_to_another}

A permutation (resp. polyomino) class being now equipped with several notions of basis,
we investigate how to describe one basis from another, and focus here on describing the $p$-basis from any $m$-basis.

\begin{proposition}
Let \C be a permutation (resp. polyomino) class, and let \M be an $m$-basis of  $\C$.
Then the $p$-basis of \C consists of all permutations (resp. polyominoes)
that contain a submatrix in \M,
and that are minimal (w.r.t. $\sympattern$ resp. $\polypattern$) for this property.
\label{prop:description_p-basis}
\end{proposition}

\begin{proof}
This follows immediately from the description of the $p$-basis of \C in 
Proposition~\ref{prop:perm_basis} (resp.~\ref{prop:polyomino_basis}) 
and the definition of $m$-basis. 
\end{proof}

\begin{example}
Figures~\ref{fig:basisRectangularClass} and~\ref{fig:basis321_231_312}
give the $p$-bases of the classes $\A$ and $\R$ of Example~\ref{ex:canonical_m-basis}, 
and illustrate their relation to their canonical $m$-basis.
\label{ex:relation_p-_and_m-basis}
\end{example}

\begin{figure}[ht]
\begin{center}
\includegraphics[width=9cm]{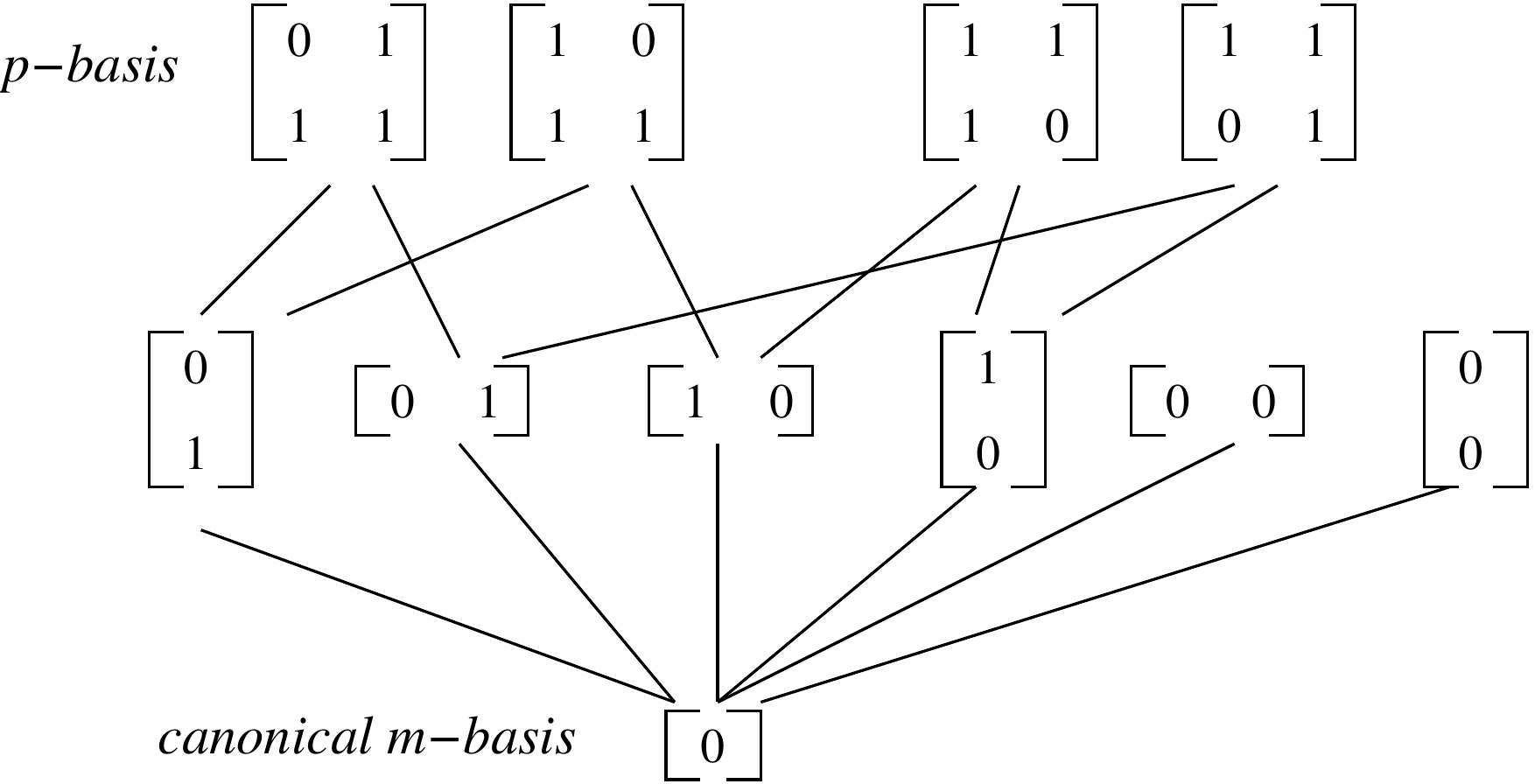}
\caption{The $p$-basis, an $m$-basis and the canonical $m$-basis of the class $\R$ of polyominoes having rectangular shape.}
\label{fig:basisRectangularClass}
\end{center}
\end{figure}

\begin{figure}[ht]
\begin{center}
\includegraphics[width=10cm]{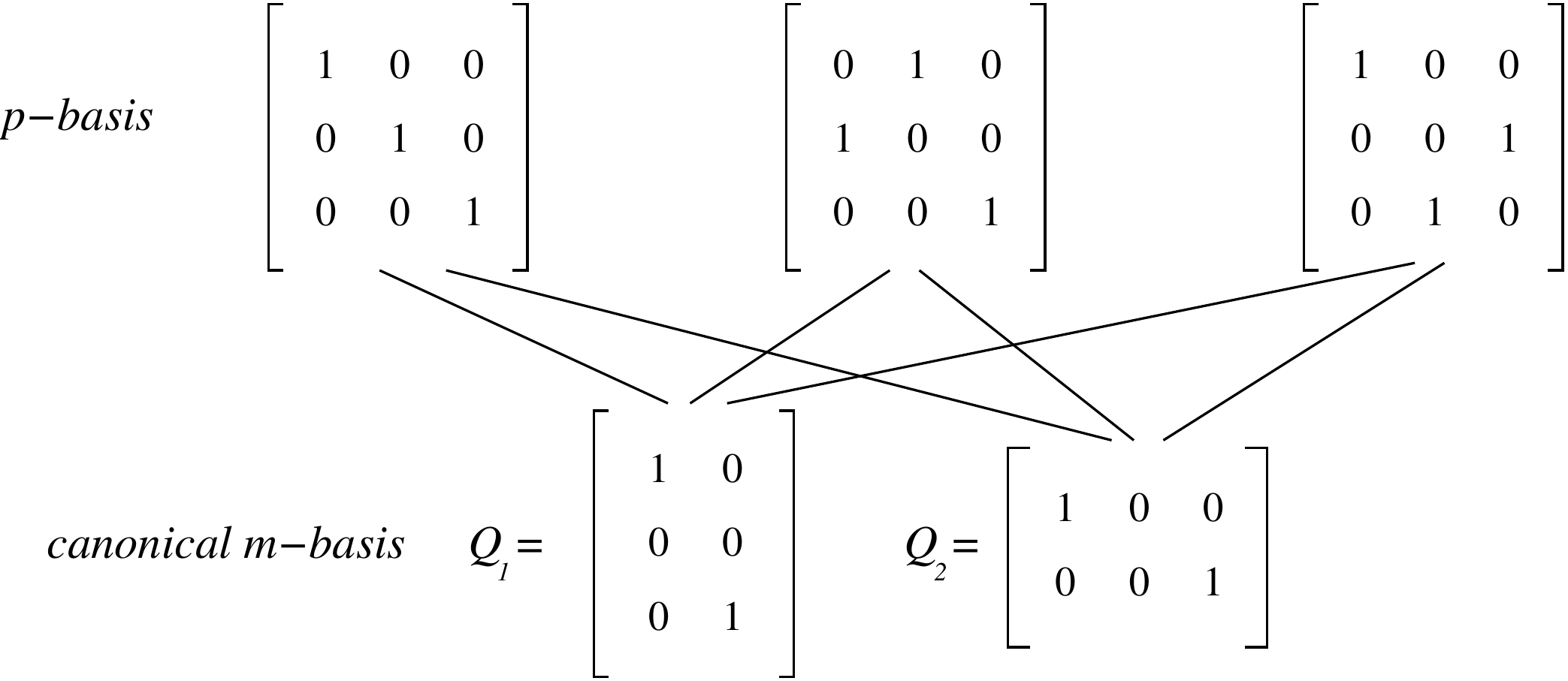}
\caption{The $p$-basis and the canonical $m$-basis of $\A = \Avperm(321,231,312)$.}
\label{fig:basis321_231_312}
\end{center}
\end{figure}  

In the case of permutation classes,
Proposition~\ref{prop:description_p-basis} allows to compute the $p$-basis of any class \C, given an $m$-basis of \C.
Indeed, the minimal permutations (in the sense of $\sympattern$) that contain a given submatrix $M$ are 
straightforward to describe:

\begin{proposition}
Let $M$ be a quasi-permutation matrix.
The minimal permutations that contain $M$ are exactly those that may be obtained from $M$
by insertions of rows (resp. columns) with exactly one entry $1$,
which should moreover fall into a column (resp. row) of $0$ of $M$.
%
\label{prop:minimal_perm_containing_M}
\end{proposition}

It follows from Proposition~\ref{prop:minimal_perm_containing_M}
that the $p$-basis of a permutation class \C can be easily computed from an $m$-basis of \C.
Also, Proposition~\ref{prop:minimal_perm_containing_M} implies that:

\begin{corollary}
If a permutation class has a finite $m$-basis
(\ie is described by the avoidance of a finite number of submatrices)
then it has a finite $p$-basis.
\label{cor:finiteness_of_m-basis_for_permutations}
\end{corollary}  

The situation is more complex if we consider polyomino classes.
The description of the polyominoes containing a given submatrix
is not as straightforward as in Proposition~\ref{prop:minimal_perm_containing_M},
and the analogue of Corollary~\ref{cor:finiteness_of_m-basis_for_permutations} does not hold for polyomino classes, 
as illustrated in Proposition~\ref{prop:infinite_basis}.

\begin{proposition}
The polyomino class $\C= \Avp(M_{inf})$ defined by the avoidance of
$
M_{inf}= {\footnotesize 
\left[\begin{array}{cccc}
1 & 0 & 0 & 1\\
1 & 1 & 0 & 1
\end{array}\right] }
$
has an infinite $p$-basis.
\label{prop:infinite_basis}
\end{proposition}

\begin{proof}
It is enough to exhibit an infinite sequence of polyominoes containing $M_{inf}$,
and that are minimal (for $\polypattern$) for this property.
By minimality of its elements, such a sequence is necessarily an antichain,
and it forms an infinite subset of the $p$-basis of $\C$.
The first few terms of such a sequence are depicted in Figure~\ref{fig:infinite},
and the definition of the generic term of this sequence should be clear from the figure.
We check by comprehensive verification that
every polyomino $P$ of this sequence contains $M_{inf}$,
and additionally that occurrences of $M_{inf}$ in $P$ always involve
the two bottommost rows of $P$, its two leftmost columns, and its rightmost column.
Moreover, comprehensive verification also shows that
every polyomino $P$ of this sequence is minimal for the condition $M_{inf} \polypattern P$,
\ie that every polyomino $P'$ occurring in such a $P$ as a proper submatrix avoids $M_{inf}$.
Indeed, the removal of rows or columns from such a polyomino $P$ either disconnects it or removes all the occurrences of $M_{inf}$.
\end{proof}

\begin{figure}[ht]
\begin{center}
\includegraphics[width=11cm]{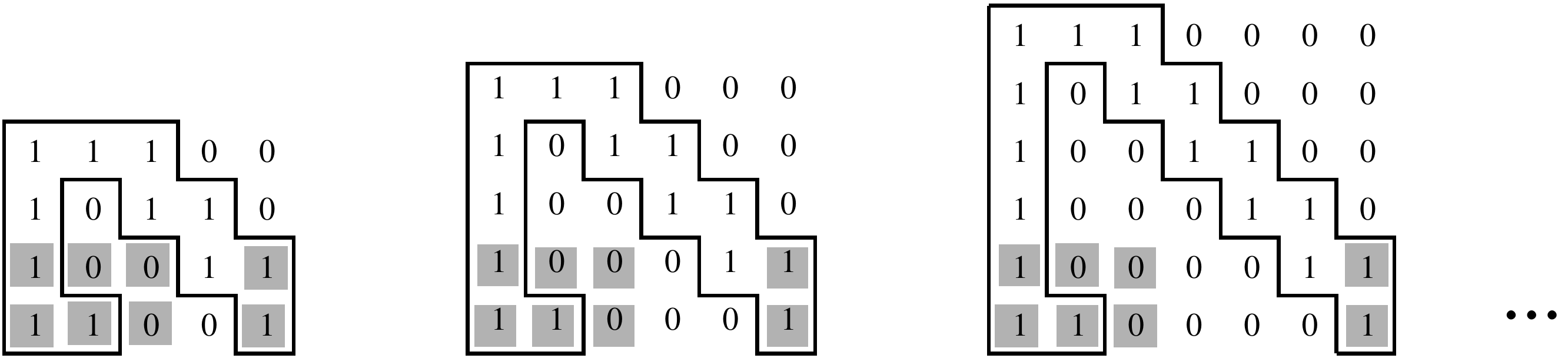}
\caption{An infinite antichain of polyominoes belonging to the $p$-basis of $\Avp(M_{inf})$.}
\label{fig:infinite}
\end{center}
\end{figure}

\section{Some permutation classes defined by submatrix avoidance}
\label{sec:known_classes_perm}

Many notions of pattern avoidance in permutations have been considered in the literature.
The one of submatrix avoidance that we have considered is yet another one.
In this section, we start by explaining how it relates to other notions of pattern avoidance.
Then, using this approach to pattern avoidance, we give simpler proofs of the enumeration of some permutation classes.
Finally, we show how these proofs can be brought to a more general level,
to prove Wilf-equivalences of many permutation classes.

\subsection{Connection of submatrix avoidance to the avoidance of other generalized permutation patterns}

A first generalization of pattern avoidance in permutations
introduces \emph{adjacency constraints} among the elements of a permutation that should form an occurrence of a (otherwise classical) pattern.
Such patterns with adjacency constraints are known as \emph{vincular} and \emph{bivincular} patterns \cite{BMCDK},
and a generalization with additional border constraints has recently been introduced  by~\cite{U}.

The avoidance of submatrices in permutations can be seen as a dual notion
to the avoidance of vincular and bivincular patterns, in the following sense.
In an occurrence of some vincular (resp. bivincular) pattern in a permutation $\sigma$,
we impose that some elements of the domain must be adjacent (resp. mapped by $\sigma$ into consecutive values). 
Now, consider a quasi-permutation matrix $M$, and let $\pi$ be the largest permutation contained in $M$. 
If there is a column (resp. row) of $0$'s in $M$, then
an occurrence of $M$ in a permutation $\sigma$ is
an occurrence of $\pi$ in $\sigma$
where some elements of the domain of $\sigma$ are not allowed to be adjacent
(resp. mapped by $\sigma$ into consecutive values).

\begin{example} 
$\Avperm \left({\footnotesize \left[ \begin{array}{cccc}
0 & 0 & 1 & 0 \\
1 & 0 & 0 & 0 \\
0 & 0 & 0 & 1 \end{array} \right]}\right)$ denotes the set of all permutations such that
in any occurrence of $231$ the elements mapped into $``2"$ and $``3"$ are at adjacent positions.

\noindent $\Avperm \left({\footnotesize\left[ \begin{array}{cccc}
0 & 0 & 1 & 0 \\
1 & 0 & 0 & 0 \\
0 & 0 & 0 & 0 \\
0 & 0 & 0 & 1 \end{array} \right]}\right)$ denotes the set of all permutations such that
in any occurrence of $231$ the elements mapped into $``2"$ and $``3"$ are at adjacent positions,
or the actual values of $``1"$ and $``2"$ are consecutive numbers.

\noindent $\Avperm \left({\footnotesize\left[ \begin{array}{ccc}
0 & 0 & 0 \\
0 & 1 & 0 \\
1 & 0 & 0 \\
0 & 0 & 1 \end{array} \right]}\right)$ denotes the set of all permutations such that
in any occurrence of $231$ the actual value of $``3"$ is the maximum of the permutation.
\end{example}

As noticed in Remark~\ref{rem:submatrix_avoidance_implies_class},
sets of permutations defined by avoidance of submatrices are permutation classes.
On the contrary, avoidance of vincular or bivincular patterns 
(although very useful for characterizing important families of permutations, like Baxter permutations~\cite{GBaxter})
does not in general describe sets of permutations that are downward closed for $\sympattern$. 
Hence, in the context of permutation \emph{classes}, 
the above discussion suggests that it is more convenient to introduce
\emph{non-adjacency} (instead of adjacency) constraints in permutation patterns, 
which correspond to rows and columns of $0$ in quasi-permutation matrices.

\emph{Mesh patterns} are another generalization of permutation patterns that has been introduced more recently by Br\"{a}nd\'{e}n and Claesson~\cite{mesh}.
It has itself been generalized in several way, in particular by \'Ulfarsson who introduced in~\cite{U} the notion of marked mesh patterns.
It is very easy to see that the avoidance of a quasi-permutation matrix with no uncovered $0$ entries\footnote{
A $0$ entry is \emph{covered} when there is a $1$ entry in its column or row (or both).}
can be expressed as the avoidance of a special form of marked-mesh pattern.
Indeed, a row (resp. $k$ consecutive rows) of $0$ in a quasi-permutation matrix with no uncovered $0$ entries
corresponds to a mark, spanning the whole pattern horizontally,
indicating the presence of at least one (resp. at least $k$) element(s).
The same holds for columns and vertical marks.
Figure~\ref{fig:marked-mesh_pattern} shows an example of a quasi-permutation matrix with no uncovered $0$ entries
with the corresponding marked-mesh pattern.

\begin{figure}[ht]
\begin{center}
\includegraphics[width=8cm]{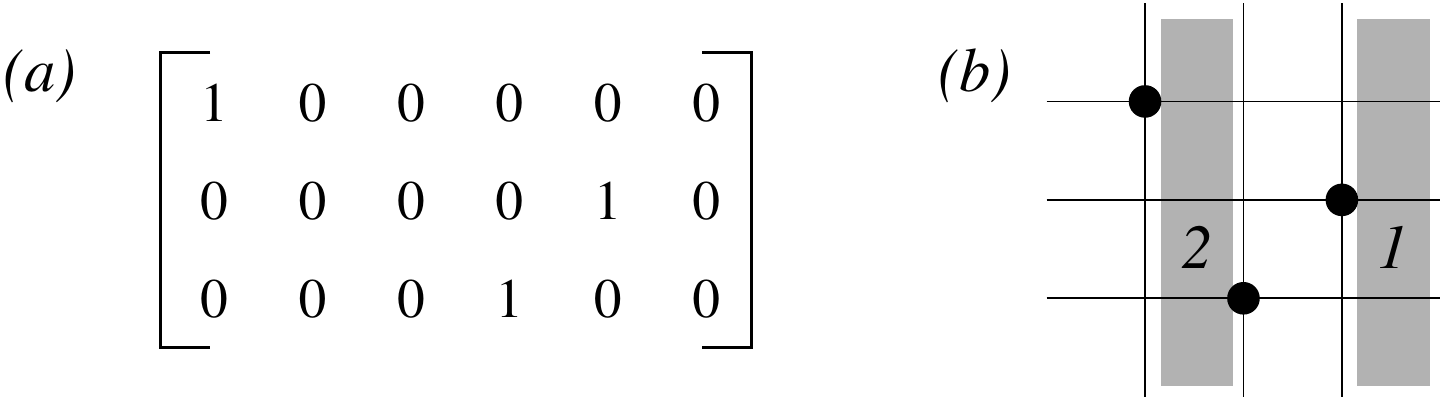}
\caption{A quasi-permutation matrix with no uncovered $0$ entries
and the corresponding marked-mesh pattern.}
\label{fig:marked-mesh_pattern}
\end{center}
\end{figure}

\subsection{A different look at some known permutation classes}

Several permutation classes avoiding three patterns of size $3$ or four patterns of size $4$ that have been studied in the literature
(and are referenced in Guibert's catalog~\cite[Appendix A]{GuibertThese})
are easier to describe with the avoidance of just one submatrix, as we explain in the following.
For some of these classes, the description by submatrix avoidance also allows to provide a simple proof of their enumeration.

In this paragraph, we do not consider classes that are equal up to symmetry (reverse, inverse, complement, and their compositions).
But the same results of course apply (up to symmetry) to every symmetry of each class considered.

\paragraph*{The class $\F = \Avperm(123,132,213)$.}
This class was studied by Simion and Schmidt~\cite{SimionSchmidt},
in the context of the systematic enumeration of permutations avoiding patterns of size $3$.
An alternative description of $\F$ is 
$
\F = \Avperm(M_F)$ where $M_F = {\footnotesize \left[ \begin{array}{ccc}
0 & 0 & 1 \\
1 & 0 & 0 \end{array} \right]}\text{.}
$
It follows immediately since $123,132$ and $213$ are exactly the permutations which cover $M_F$
in the sense of Propositions~\ref{prop:description_p-basis} and~\ref{prop:minimal_perm_containing_M}.

\cite{SimionSchmidt} shows that $\F$ is enumerated by the Fibonacci numbers.
Of course, it is possible to use the description of $\F$ by the avoidance of $M_F$ 
to prove this enumeration result. 
However, in this case the proof would just rephrase the original one of~\cite{SimionSchmidt}.

\paragraph*{The class $\G = \Avperm(123,132,231)$.}
This class is also studied in~\cite{SimionSchmidt}, where it is shown that there are $n$ permutations of size $n$ in $\G$.
The enumeration is obtained by a simple inductive argument,
which relies on a recursive description of the permutations in $\G$.

For the same reasons as in the case of $\F$, $\G$ is alternatively described by
$
\G = \Avperm(M_G)$ where $M_G = {\footnotesize\left[ \begin{array}{ccc}
0 & 1 & 0 \\
1 & 0 & 0 \end{array} \right]}\text{.}
$
Any occurrence of a pattern $12$ in a permutation $\sigma$ can be extended to an occurrence of $M_G$,
as long as it does not involve the last element of $\sigma$.
So from this characterization, it follows that the permutations of $\G$ are exactly the decreasing sequences followed by one element.
This describes the permutations of $\G$ non-recursively, and gives immediate access to the enumeration of $\G$.

\paragraph*{The classes $\H = \Avperm(1234,1243,1423,4123)$, $\J = \Avperm(1324,1342,1432,4132)$ and $\K = \Avperm(2134,2143,2413,4213)$.}

These three classes have been studied in~\cite[Section 4.2]{GuibertThese},
where it is proved that they are enumerated by the central binomial coefficients.
The proof first gives a generating tree for these classes,
and then the enumeration is derived analytically from the corresponding rewriting system.
In particular, this proof does not provide a description of the permutations in $\H$, $\J$ and $\K$
which could be used to give a \emph{combinatorial} proof of their enumeration.
Excluded submatrices can be used for that purpose.

As before, because the $p$-basis of $\H$, $\J$ and $\K$ are exactly the permutations which cover the matrices $M_H$, $M_J$ and $M_K$ given below,
introducing the matrices 
$M_H = {\footnotesize\left[ \begin{array}{ccc}
0 & 0 & 0 \\
0 & 0 & 1 \\
0 & 1 & 0 \\
1 & 0 & 0 \end{array} \right]}$, 
$M_J = {\footnotesize\left[ \begin{array}{ccc}
0 & 0 & 0 \\
0 & 1 & 0 \\
0 & 0 & 1 \\
1 & 0 & 0 \end{array} \right]}$ 
and 
$M_K = {\footnotesize\left[ \begin{array}{ccc}
0 & 0 & 0  \\
0 & 0 & 1  \\
1 & 0 & 0  \\
0 & 1 & 0 \end{array} \right]}$, 
we have $\H = \Avperm(M_H)$, $\J = \Avperm(M_J)$ and $\K = \Avperm(M_K)$. 

Similarly to the case of $\G$, any occurrence of a pattern $123$ (resp. $132$, resp. $213$)
in a permutation $\sigma$ can be extended to an occurrence of $M_H$ (resp. $M_J$, resp. $M_K$),
as long as it does not involve the maximal element of $\sigma$.
Conversely, if a permutation $\sigma$ contains $M_H$ (resp. $M_J$, resp. $M_K$),
then there is an occurrence of $123$ (resp. $132$, resp. $213$) in $\sigma$ that does not involves its maximum.
Consequently, the permutations of $\H$ (resp. $\J$, resp. $\K$) are exactly those of avoiding $123$ (resp. $132$, resp. $213$)
to which a maximal element has been added.
This provides a very simple description of the permutations of $\H$, $\J$ and $\K$.
Moreover, recalling that for any permutation $\pi \in \sym_3$ $\Avperm(\pi)$ is enumerated by the Catalan numbers,
it implies that the number of permutations of size $n$ in $\H$ (resp. $\J$, resp. $\K$) is $n\times Cat_{n-1} = {{2n-2}\choose {n-1}}$.

\subsection{Propagating enumeration results and Wilf-equivalences with submatrices}

The similarities that we observed between the cases of the classes $\G$, $\H$, $\J$ and $\K$ are not a coincidence.
Indeed, they can all be encapsulated in the following proposition, which simply pushes the same idea forward to a general setting.

\begin{proposition}
Let $\tau$ be a permutation.
Let $M_{\tau,top}$ (resp. $M_{\tau,bottom}$) be the quasi-permutation matrix obtained by adding
a row of $0$ entries above (resp. below) the permutation matrix of $\tau$.
Similarly, let $M_{\tau,right}$ (resp. $M_{\tau,left}$) be the quasi-permutation matrix obtained by adding
a column of $0$ entries on the right (resp. left) of the permutation matrix of $\tau$.
The permutations of $\Avperm(M_{\tau,top})$ (resp. $\Avperm(M_{\tau,bottom})$, resp. $\Avperm(M_{\tau,right})$, resp. $\Avperm(M_{\tau,left})$)
are exactly the permutations avoiding $\tau$ to which a maximal (resp. minimal, resp. last, resp. first) element has been added.
\label{prop:perm+0_avoidance}
\end{proposition}

\begin{proof}
We prove the case of $M_{\tau,top}$ only, the other cases being identical up to symmetry. 

Consider a permutation $\sigma \in \Avperm(M_{\tau,top})$, and denote by $\sigma'$ the permutation obtained deleting the maximum of $\sigma$.
Assuming that $\sigma'$ contains $\tau$, then $\sigma$ would contain $M_{\tau,top}$, which contradicts $\sigma \in \Avperm(M_{\tau,top})$; hence $\sigma' \in \Avperm(\tau)$.

Conversely, consider a permutation $\sigma' \in \Avperm(\tau)$, and a permutation $\sigma$ obtained by adding a maximal element to $\sigma'$.
Assume that $\sigma$ contains $M_{\tau,top}$, and consider an occurrence of $M_{\tau,top}$ in $\sigma$.
This occurrence of $M_{\tau,top}$ cannot involve the maximum of $\sigma$, so
it yields an occurrence of $\tau$ in $\sigma'$, and hence a contradiction.
Therefore, $\sigma \in \Avperm(M_{\tau,top})$.
\end{proof}

Proposition~\ref{prop:perm+0_avoidance} has two nice consequences in terms of enumeration.
When the enumeration of the class is known,
it allows to deduce the enumeration of four other permutation classes.
Similarly, for any pair of Wilf-equivalent classes (\emph{i.e.} of permutation classes having the same enumeration),
it produces a set of eight classes which are all Wilf-equivalent. 

\begin{corollary}
Let $\C$ be a permutation class whose $p$-basis is $\B$.
Let $\M = \{M_{\tau,top} : \tau \in \B\}$.
Denote by $c_n$ the number of permutations of size $n$ in $\C$.
The permutation class $\Avperm(\M)$ is enumerated by the sequence $(n \cdot c_{n-1})_n$.
The same holds replacing $M_{\tau,top}$ with $M_{\tau,bottom}$, $M_{\tau,right}$ or $M_{\tau,left}$.
\end{corollary}

\begin{corollary} \label{cor:WE}
Let $\C_1$ and $\C_2$ be two Wilf-equivalent permutation classes whose $p$-basis are respectively $\B_1$ and $\B_2$.
Let $\M_{1,top} = \{M_{\tau,top} : \tau \in \B_1\}$, $\M_{2,top} = \{M_{\tau,top} : \tau \in \B_2\}$, 
and define $\M_{1,bottom}$, $\M_{2,bottom}$, $\M_{1,right}$, $\M_{2,right}$, $\M_{1,left}$, $\M_{2,left}$ similarly. 
The permutation classes defined by the avoidance of one of these eight sets of matrices are all Wilf-equivalent. 
\end{corollary}

\section{Submatrix avoidance in polyominoes}
\label{sec:known_classes_poly}

In this section, we show that several families of polyominoes studied in the literature can be characterized in terms of submatrix avoidance. For more details on these families of polyominoes we address the reader to \cite{mbm,DV}. We also use submatrix avoidance to introduce new polyomino classes. 

\subsection{Some known polyomino classes}

Quite a few known families of polyominoes defined by geometric constraints (see Figure~\ref{fig:polyominoes}) can be described 
by the avoidance of submatrices (which translate these geometric constraints in a natural way). 

\begin{figure}[ht]
\begin{center}
\includegraphics[width=10cm]{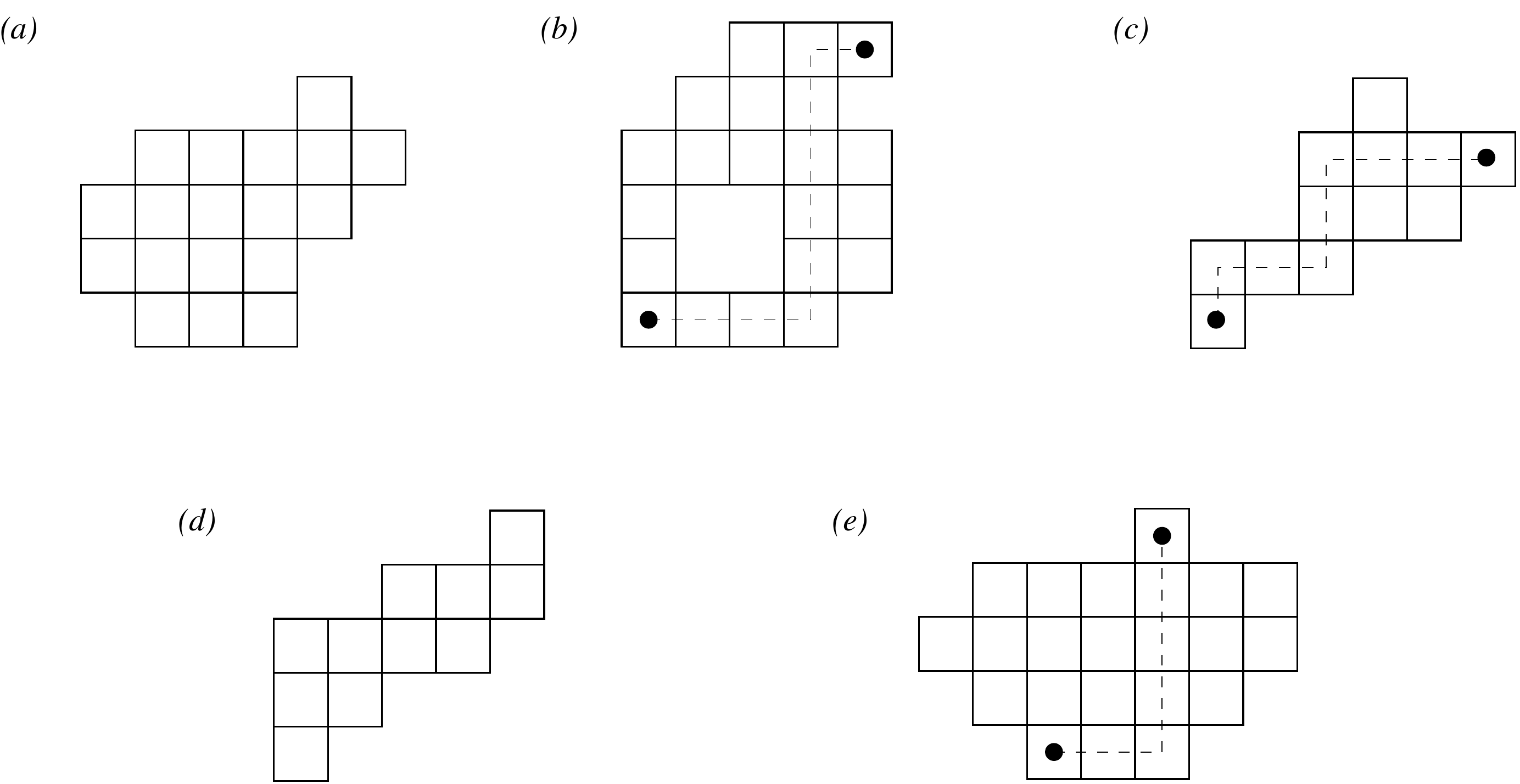}
\caption{$(a)$ a convex polyomino; $(b)$ a directed polyomino; $(c)$ a directed-convex polyomino; $(d)$ a parallelogram polyomino; $(e)$ an $L$-convex polyomino.}
\label{fig:polyominoes}
\end{center}
\end{figure}

\paragraph{Convex polyominoes.} These are defined by imposing one of the simplest geometrical constraints: the connectivity of rows/columns.
\begin{definition}
A polyomino is {\em horizontally convex} or {\em $h$-convex} (resp. {\em vertically convex} or {\em $v$-convex}), if each of its rows (resp. columns) is connected. A polyomino that is both $h$-convex and $v$-convex is said to be {\em convex}.
\end{definition}

Figure~\ref{fig:polyominoes}$(a)$ shows a convex polyomino. 
The convexity constraint can be easily expressed in terms of excluded submatrices since its definition already relies on some specific configurations that the cells of each row and column of a polyomino have to avoid: 

\begin{proposition}
\label{prop:m-basis_convex}
Convex polyominoes can be represented by the avoidance of the two submatrices 
$H = {\footnotesize \left[ \begin{array}{ccc}
1 & 0 & 1 \end{array} \right]}$ and $V = {\footnotesize\left[ \begin{array}{c}
1 \\
0 \\
1 \end{array} \right]}$.
\end{proposition}

More precisely, the avoidance of the matrix $H$ (resp. $V$)
indicates the $h$-convexity (resp. $v$-convexity). 

\begin{remark}
It is easily checked from Definition~\ref{def:minimal_m-basis} that $\{H,V\}$ is a minimal $m$-basis of the class of convex polyominoes. 
By Proposition~\ref{prop:minimal_m-basis}, this implies that $\{H,V\}$ is included in the canonical $m$-basis of this class. 
We would like to point out that this inclusion is strict, since the matrix 
$A = {\footnotesize \left[ \begin{array}{ccccc}
0 & 0 & 1 \\
1 & 0 & 0 \\
0 & 1 & 0  
 \end{array} \right]}$ 
also belongs to this canonical $m$-basis. 
Indeed, $(i)$ there is no convex polyomino that contains~$A$, and 
$(ii)$ every proper submatrix of~$A$ is included in some convex polyomino. 
The proof of $(ii)$ is immediate by comprehensive verification, 
and $(i)$ follows easily by noticing that every polyomino containing~$A$ must have some rows or columns which are not connected. 
Finding the canonical $m$-basis of the class of convex polyominoes remains an open problem. 
%
\end{remark}

From Proposition~\ref{prop:description_p-basis}, it is also easy to determine the $p$-basis of the class of convex polyominoes: 
it is the set of four polyominoes $\{H_1,H_2,V_1,V_2\}$ depicted in Figure~\ref{fig:convex_p-basis}.

\begin{figure}[ht]
\begin{center}
\includegraphics[width=11cm]{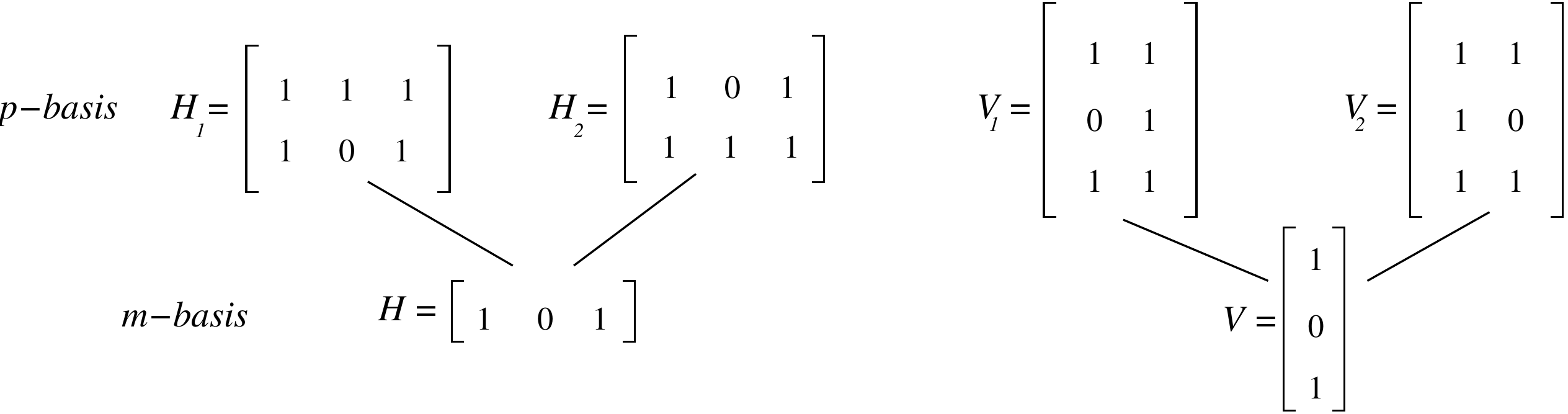}
\caption{The $p$-basis and a minimal $m$-basis of the class of convex polyominoes.}
\label{fig:convex_p-basis}
\end{center}
\end{figure}

We advise the reader that, in the rest of the section, for each polyomino class, we will provide only a matrix description of the basis. Indeed, in all these examples the $p$-basis can easily be obtained from the given $m$-basis with Proposition~\ref{prop:description_p-basis}, like in the case of convex polyominoes. 

\paragraph{Directed-convex polyominoes.}

Directed-convex polyominoes are defined using the notion of internal path of a polyomino. An {\em (internal) path} of a polyomino is a sequence of distinct cells $(c_1,\dots,c_n)$ of the polyomino such that every two consecutive cells in this sequence are edge-connected; according to the respective positions of the cells $c_i$ and $c_{i+1}$, we say that the pair $(c_i,c_{i+1})$ forms a {\em north, south, east or west step} in the path.

\begin{definition}
A polyomino $P$ is {\em directed} when every cell of $P$ can be reached
from a distinguished cell  (called the source) by a path that uses only north and east steps.
\end{definition}

Figure~\ref{fig:polyominoes}$(b)$ depicts a directed polyomino.
The reader can check 
(\emph{e.g.}, removing the fourth column on Figure~\ref{fig:polyominoes}$(b)$) 
that the set of directed polyominoes is not a polyomino class.
However, the set of \emph{directed-convex} polyominoes
(\ie of polyominoes that are both directed and convex -- see Figure~\ref{fig:polyominoes}$(c)$)
is a class:

\begin{proposition}
The class ${\cal D}$ of directed-convex polyominoes is characterized by the avoidance of the submatrices
$H= {\footnotesize
  \left[\begin{array}{ccc}
          1 & 0 & 1
         \end{array}
  \right]}$,
$V={\footnotesize \left[\begin{array}{c}
          1\\
          0\\
          1
         \end{array}
  \right]}$
and $  D={\footnotesize \left[\begin{array}{cc}
          1 & 1\\
          0 & 1
         \end{array}
  \right]}$.
\label{prop:m-basis_directed_convex}
\end{proposition}

\begin{proof}
Let us prove that ${\cal D}=\Avp (H,V,D)$.

First, let $P\in \Avp (H,V,D)$. 
By Proposition~\ref{prop:m-basis_convex}, $P$ is a convex polyomino which avoids the submatrix $D$. 
Let $s$ be the bottom cell of the leftmost column of $P$. 
If $P$ is directed, then necessarily $s$ is its source.
Let us proceed by contradiction assuming that $P$ is not directed. 
Then there exists a cell $c$ of $P$ such that all the paths from $s$ to $c$ contain either a south step or a west step. 
Consider one of these paths $p$ with minimal length (where the length of a path is just the number of its cells), say $\ell$, 
and having at least a south step (if $p$ has at least a west step, a similar reasoning holds). 
Let us denote by $i$, with $1\leq i < \ell$ the index such that the pair of cells $(c_i,c_{i+1})$ forms the first south step in $p$. 
Let us consider two cases.
\begin{description}
\item{i.)} If the first vertical step of $p$ is a north step, 
then let $(c_j,c_{j+1})$ be the last north step before the first south step of $p$. 
By minimality of $p$, it holds that   
$c_{j+1}$ and $c_i$ are two distinct cells on the same row, that do not share an edge. 
Consequently, $c_j$ and $c_{i+1}$ are also distinct cells on the same row that do not share an edge. 
And by minimality of $p$, this row is not connected, contradicting the $h$-convexity of $P$. 
\item{ii.)} If the first vertical step of $p$ is a south step, \ie 
if $(c_i,c_{i+1})$ is the first vertical step of $p$, 
it means that $s$ and $c_i$ are on the same row of $P$. 
The cell $q$ immediately below $s$ is then on the row of $c_{i+1}$, 
and by definition of $s$, it does not belong to $P$. 
Therefore, the four cells $s$, $q$, $c_i$, and $c_{i+1}$ form the pattern $D$, a contradiction.
\end{description}

Conversely, let $P$ be a directed convex polyomino. 
By Proposition~\ref{prop:m-basis_convex}, $P$ avoids $H$ and $V$. 
We will proceed by contradiction assuming that $P$ contains the pattern $D$. 
Let $c_1$ and $c_2$ be the cells of $P$ that correspond to the upper left and to the lower right cells of $D$, respectively. 
The source cell $s$ of $P$ must lie weakly below (resp. weakly to the left of) $c_2$, 
or every path from $s$ to $c_2$ would have to contain at least one south (resp. west) step, contradicting the fact that $P$ is directed. 
Now consider any path leading from $s$ to $c_1$. 
Either it runs entirely on the left of $c_1$ (in the weak sense), so that $P$ contains the pattern $H$ or $V$, 
which is against the convexity of $P$. 
Or it contains a west step, which again contradicts the fact that $P$ is directed.
\end{proof}

\paragraph{Parallelogram polyominoes.}
Another widely studied family of polyominoes --that can also be defined using a notion of path, this time of {\em boundary path}-- is that of {\em parallelogram polyominoes} (see Figure~\ref{fig:polyominoes}$(d)$).

\begin{definition} A parallelogram polyomino is a polyomino whose boundary can be decomposed in two paths, the upper and the lower paths, which are made of north and east unit steps and meet only at their starting and final points.
\end{definition}

Proposition~\ref{prop:parallelogram} below shows that parallelogram polyominoes can be described by submatrix avoidance, 
hence form a polyomino class. 
Since the proof of Proposition~\ref{prop:parallelogram} resembles that of Proposition~\ref{prop:m-basis_directed_convex}, it will be left to the reader.

\begin{proposition}\label{prop:parallelogram}
Parallelogram polyominoes are characterized by the avoidance of the submatrices 
$ {\footnotesize
   \left[\begin{array}{cc}
          1 & 0\\
          1 & 1
         \end{array}
  \right]}$ and ${\footnotesize
  \left[\begin{array}{cc}
          1 & 1\\
          0 & 1
         \end{array}
  \right]}$. 
\label{prop:m-basis_parall}
\end{proposition}

\begin{remark}
The set of two excluded submatrices of Proposition~\ref{prop:m-basis_parall} 
is also the $p$-basis of the class of parallelogram polyominoes. 
And it can further be checked that it forms a minimal $m$-basis. 
Therefore, by Proposition~\ref{prop:when_p-basis_is_minimal}, 
it is the unique minimal $m$-basis of the class of parallelogram polyominoes. 
\end{remark}

\paragraph{$L$-convex polyominoes.}
Parallelogram polyominoes and directed-convex polyominoes form subclasses of the class of convex polyominoes 
that are both defined in terms of paths. 
The relationship between paths and convexity is closer than it may appear. 
In \cite{CR} the authors observed that convex polyominoes have the property that every pair of cells
is connected by a monotone path, and proposed a classification of convex polyominoes based on
the number of changes of direction in the paths connecting any two cells of a polyomino. More
precisely, a convex polyomino is {\itshape $k$-convex} if every pair of its cells can be connected by a monotone path
with at most $k$ changes of direction, and $k$ is called the {\itshape convexity degree} of the polyomino.

The $1$-convex polyominoes are more commonly called $L$-convex polyominoes: any two cells can be connected by a path with at most one change of direction (see Figure~\ref{fig:polyominoes}$(e)$). In recent literature $L$-convex polyominoes have been considered from several points of view: under tomographical aspects~\cite{CFRR} and from the enumeration perspective~\cite{lconv,CFRR2}.

Here, we study how the constraint of being $k$-convex can be represented in terms of submatrix avoidance.
In order to reach this goal, let us present some basic definitions and properties from the field of {\em discrete tomography} \cite{ryser}. Given a binary matrix, the vector of its {\em horizontal} (resp. {\em vertical}) {\em projections} is the vector of the row (resp. column) sums of its elements. 
In 1963 Ryser~\cite{ryser} established a fundamental result which, using our notation, can be reformulated as follows:

\begin{thm}\label{th:ryser}
A binary matrix is uniquely determined by its horizontal and vertical projections if and only if it does not contain $S_1= {\footnotesize \left[\begin{array}{cc}
          1 & 0\\
          0 & 1
         \end{array}
  \right]}$ and $S_2 = {\footnotesize \left[\begin{array}{cc}
          0 & 1\\
          1 & 0
         \end{array}
  \right]}$ as submatrices.
\end{thm}

Now we show that the set $\cal L$ of $L$-convex polyominoes is a polyomino class. 

\begin{proposition}\label{prop:lconvex}
$L$-convex polyominoes are characterized by the avoidance of the submatrices $H,V,S_1$ and $S_2$. 
In other words, ${\cal L}=\Avp(H,V,S_1,S_2)$.
\label{prop:m-basis_lconvex}
\end{proposition}

\begin{proof}
First, let $P \in \Avp(H,V,S_1,S_2)$. 
By Proposition~\ref{prop:m-basis_convex}, $P$ is a convex polyomino that avoids the submatrices $S_1$ and $S_2$. 
Let us proceed by contradiction assuming that $P$ is not $L$-convex. 
It means that there exists a pair of cells $c_1$ and $c_2$ of $P$, 
that are neither horizontally nor vertically aligned, 
and such that all the paths from $c_1$ to $c_2$ are not $L$-paths. 
In other words, all paths from $c_1$ to $c_2$ have at least two changes of directions. 
Up to symmetry, we may assume that $c_2$ lies below $c_1$ and on its right. 
Consider the path from $c_1$ that goes always right (resp. down) until it encounters a cell $c'$ (resp. $c''$) which does not belong to $P$. 
This happens before reaching the column (resp. row) of $c_2$, otherwise by convexity there would be a path from $c_1$ to $c_2$ with only one change of direction. 
By convexity, all the cells on the right of $c'$ in the same row (resp. below $c''$ in the same column) do not belong to $P$. 
Consequently, the $2\times 2$-submatrix of $P$ corresponding to the two rows and two columns of $c_1$ and $c_2$ is an occurrence of the pattern $S_1$, giving the desired contradiction.

Conversely, let $P$ be an $L$-convex polyomino. 
By convexity, $P$ does not contain the submatrices $H$ and $V$. 
Moreover, in \cite{CFRR} it is proved that an $L$-convex polyomino is uniquely determined by its horizontal and vertical projections, 
so by Theorem~\ref{th:ryser} it cannot contain $S_1$ nor $S_2$.  
\end{proof}


\paragraph{$2$-convex polyominoes.}

Unlike $L$-convex polyominoes, 
$2$-convex polyominoes do not form a polyomino class. Indeed, 
 the $2$-convex polyomino in Figure~\ref{fig:2pc3}$(a)$ contains the
$3$-but-not-$2$-convex polyomino $(b)$ as a submatrix. 
Similarly, the set of $k$-convex polyominoes is not a polyomino class, for $k \geq 2$.

\begin{figure}[htp]
\begin{center}
\includegraphics[width=7cm]{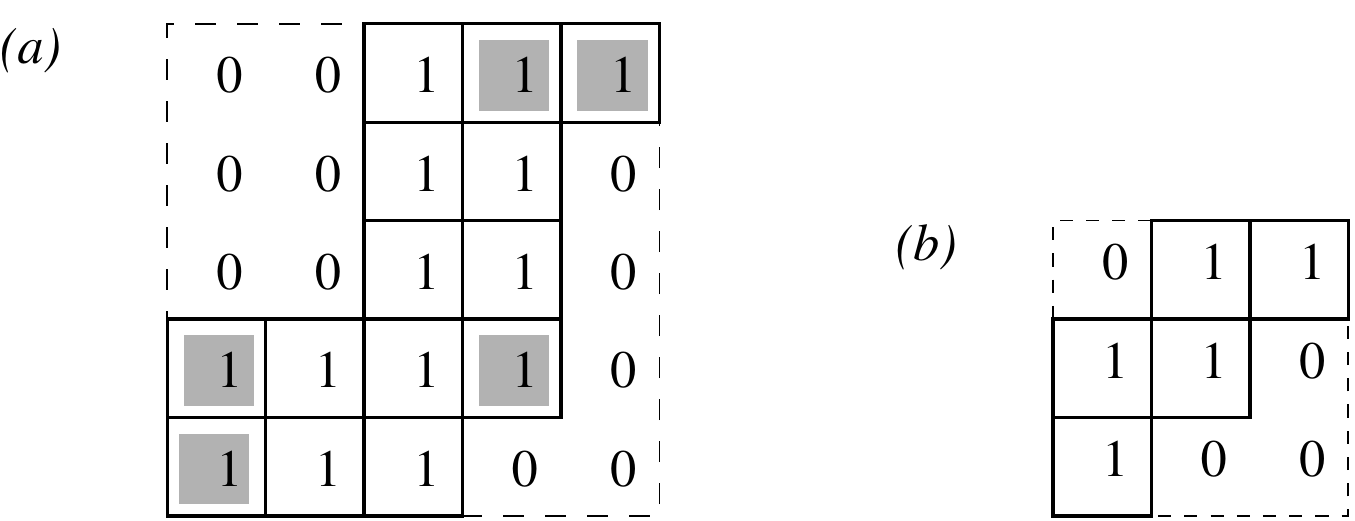}
\caption{$(a)$ a $2$-convex polyomino $P$; $(b)$ a submatrix of $P$
that is a polyomino of convexity degree $3$.} \label{fig:2pc3}
\end{center}
\end{figure}

As already pointed out, there are several other families of polyominoes that
are not polyomino classes, like directed polyominoes or polyominoes without holes. 
Consequently, these families cannot be expressed in terms of submatrix avoidance. 
This suggests to extend the notion of submatrix avoidance, in order to be able to represent more families of polyominoes. 
{\em Generalized submatrix avoidance} would impose adjacency constraints between rows or columns, 
in the same fashion as the avoidance of vincular and bivincular patterns~\cite{bona} in permutations. 
For instance, $2$-convex polyominoes can be characterized using such generalized excluded submatrices, 
although we omit the proof here, since we believe it goes beyond the aims of this paper.
We refer the reader to~\cite[\S 3.6.2]{TesiDaniela} for more details about generalized excluded submatrices. 

\subsection{Defining new polyomino classes by submatrix avoidance}

In addition to characterizing known classes, the approach of submatrix avoidance may be used to define new classes of polyominoes, the main question being then to give a combinatorial/geometrical characterization of these classes.  We present some examples of such classes, with simple characterizations and interesting combinatorial properties. These examples illustrate that the submatrix avoidance approach in the study of families of polyominoes is promising.

\paragraph{$L$-polyominoes.} Proposition~\ref{prop:lconvex}
states that $L$-convex polyominoes can be characterized
by the avoidance of four matrices: $H$ and $V$, which impose the 
convexity constraint; and $S_1$ and $S_2$, which account for the
$L$-property, or equivalently (by Theorem~\ref{th:ryser}) indicate 
the uniqueness of the polyomino w.r.t its horizontal and vertical
projections. So, it is quite natural to study the
class $\Avp(S_1,S_2)$, which we call the class of {\em $L$-polyominoes}. 
From Theorem~\ref{th:ryser}, it follows that: 

\begin{proposition}
Every $L$-polyomino is uniquely determined by its horizontal and vertical projections.
\end{proposition}

From a geometrical point of view, the $L$-polyominoes can be
characterized using the concept of (geometrical) inclusion between rows (resp.
columns) of a polyomino. 
For any polyomino $P$ with $n$ columns, and any rows $r_1=(r_{1;1} \ldots r_{1;n})$, $r_2=(r_{2;1} \ldots r_{2;n})$ of the matrix representing $P$, we say that $r_1$ is {\em geometrically included} in $r_2$ (denoted $r_1 \leqslant r_2$) if, for all $1 \leq i \leq n$ we have that $r_{1;i}=1$ implies $r_{2;i}=1$. Geometric inclusion of columns is defined analogously. 
Two rows (resp. columns) $r_1, r_2$ (resp. $c_1, c_2$) of a polyomino $P$ are said to be {\em comparable} if $r_1\leqslant r_2$ or $r_2\leqslant r_1$ (resp. $c_1\leqslant c_2$ or $c_2\leqslant c_1$). These definitions are illustrated in Figure~\ref{bicentered}. 

The avoidance of $S_1$ and $S_2$ has an immediate interpretation in geometric terms, proving that: 

\begin{proposition}
The class of $L$-polyominoes coincides with the set of the
polyominoes where every pair of rows (resp. columns) are comparable.
\end{proposition}

\begin{figure}[htd]
\begin{center}
\includegraphics[width=9cm]{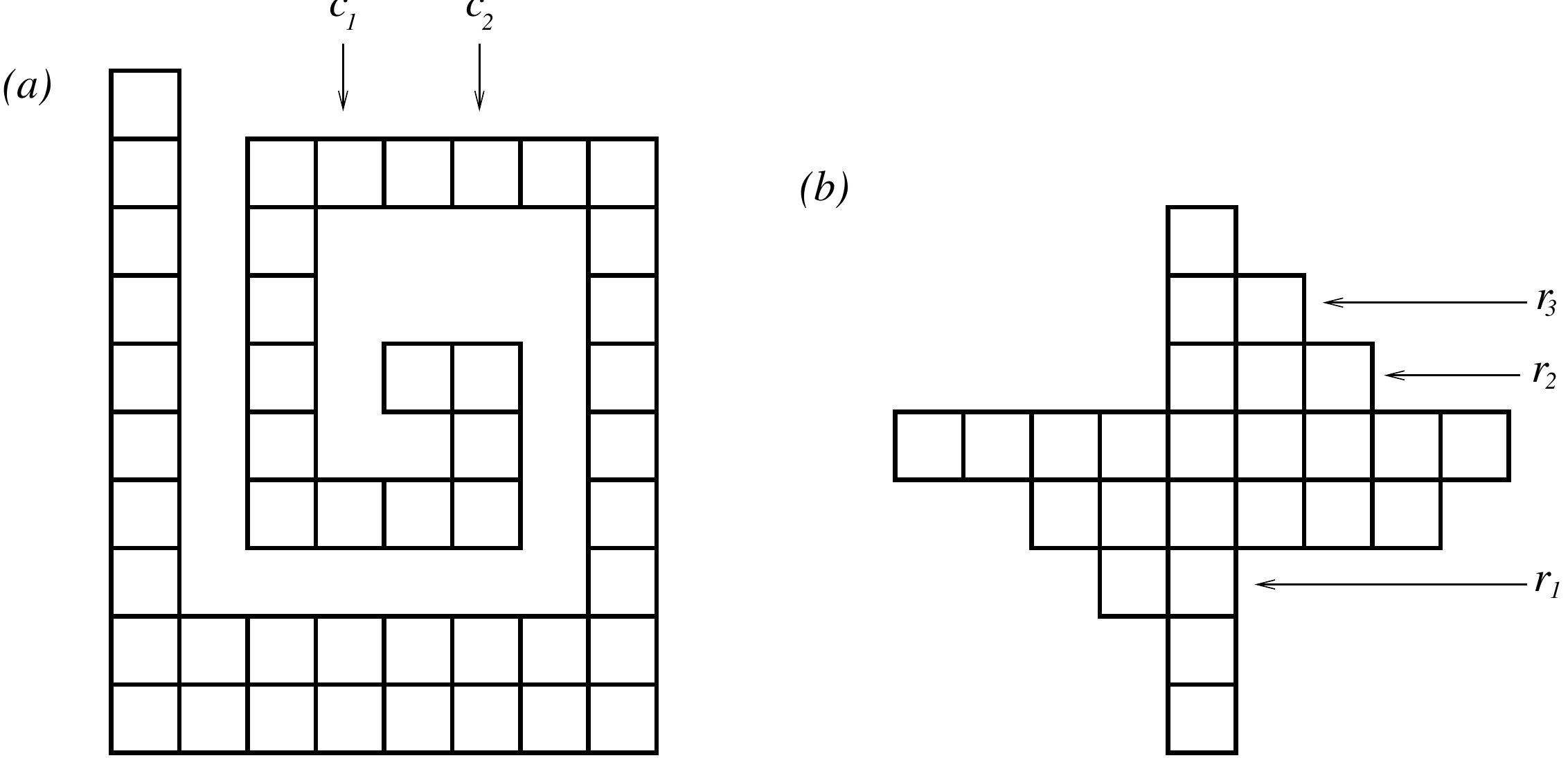}
\caption{$(a)$ a $L$-polyomino, where the reader can check that
every pair of  rows and columns are comparable; for instance,
$c_1\leqslant c_2$; $(b)$ an example of a polyomino which is not an
$L$-polyomino, where row $r_1$ is comparable 
neither to row $r_2$ nor to row $r_3$.} \label{bicentered}
\end{center}
\end{figure}

We leave open the problem of studying further the class of
$L$-polyominoes, in particular from an enumerative point of view 
(enumeration w.r.t. the area or the semi-perimeter).

\paragraph*{ The class $\Avp(H', V') $ with $H' = {\footnotesize \left[\begin{array}{c}
          0 \\
          1 \\
          0 
         \end{array}
  \right] }$ and $V' = {\footnotesize\, \left[\begin{array}{ccc}
          0  & 1 & 0 
         \end{array}
  \right] }$.} ~ \\
  
By analogy\footnote{which essentially consists in exchanging $0$ and $1$ in the excluded submatrices} with the class of convex polyominoes (characterized by the avoidance of $H$ and $V$), we may consider the class 
${\cal C}'$ of polyominoes avoiding the two submatrices $H'$ and $V'$ defined above. 
In the sense of the $0/1$ duality, these objects can be viewed as a dual class to convex polyominoes. Figure \ref{fig:010}$(a)$ shows a polyomino in ${\cal C}'$.

\begin{figure}[htd]
\begin{center}
\includegraphics[width=9cm]{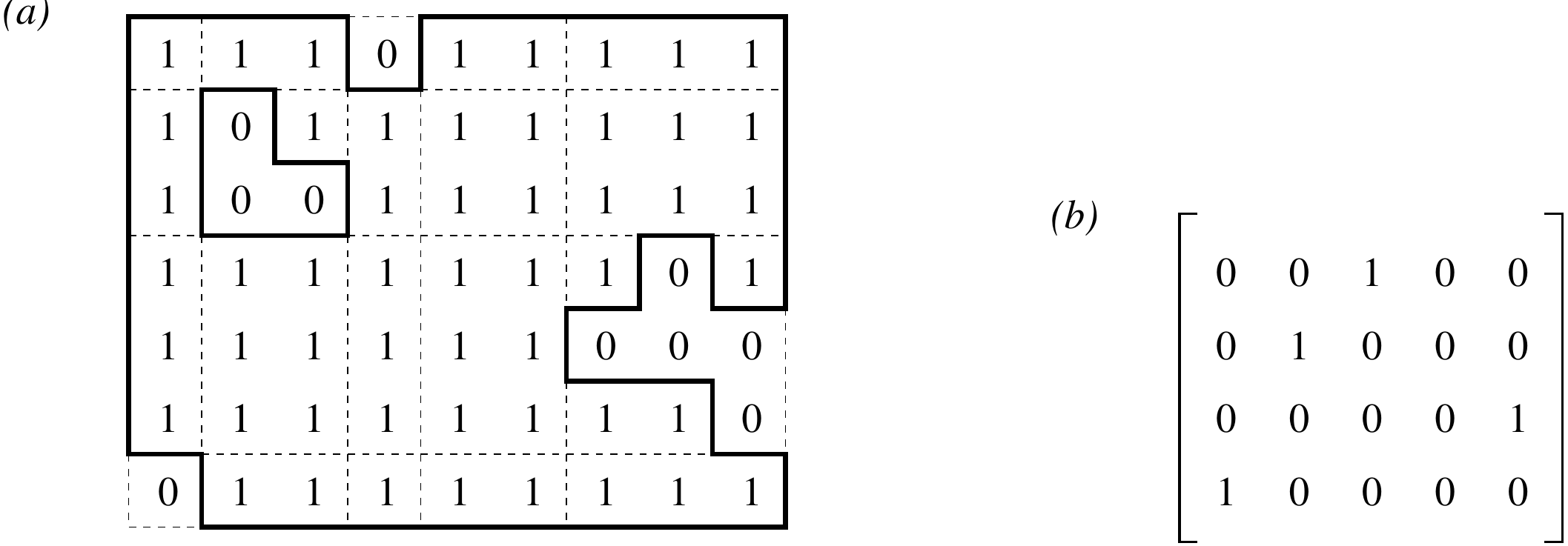}
\caption{$(a)$ a polyomino in ${\cal C}'$ and its decomposition; $(b)$ the corresponding quasi-permutation matrix.}
\label{fig:010}
\end{center}
\end{figure}

The avoidance of $H'$ and $V'$ has a straightforward geometric interpretation, giving immediately that: 

\begin{proposition}
A polyomino $P$ belongs to ${\cal C}'$ if and only if every connected set of cells of maximal length in a row (resp. column) has a contact with the minimal bounding rectangle of $P$. 
\label{prop:contact}
\end{proposition}

The avoidance of $H'$ and $V'$ also ensures that in a polyomino of ${\cal C}'$ every connected set of $0$s has the shape of a convex polyomino, which we call -- by abuse of notation -- a \emph{convex $0$-polyomino} (contained) in $P$. Each of these convex  $0$-polyominoes has a minimal bounding rectangle, which  individuates an horizontal (resp. vertical) strip of cells in $P$, where no other convex $0$-polyomino of $P$ can be found. 
Therefore every polyomino $P$ of ${\cal C}'$ can be uniquely decomposed  in regions of two types: rectangles all made of $1$s (of type $A$) or rectangles bounding a convex $0$-polyomino (of type $B$). Then, we can map $P$ onto a quasi-permutation matrix as follows: each rectangle of type $A$ is mapped onto a $0$, and each rectangle of type $B$ is mapped onto a $1$. See an example in Figure~\ref{fig:010}$(b)$. 

Although this representation is clearly non unique, we believe it may be used for the enumeration of ${\cal C}'$. 
For a start, it provides a simple lower bound on the number of polyominoes in ${\cal C}'$ whose bounding rectangle is a square. 

\begin{proposition}\label{prop:stanley_wilf} 
Let $c'_n$ be the number of polyominoes in ${\cal C}'$ whose bounding rectangle is an $n\times n$ square. 
For $n\geq 1$, $c'_n \geq \lfloor \frac{n}{2} \rfloor \, ! \, .$
\end{proposition}

\begin{proof}
The statement directly follows from a mapping from permutations of size  $m \geq 1$ to polyominoes in ${\cal C}'$ whose bounding rectangle is an $2m\times 2m$ square, and defined as follows.  
From a permutation $\pi$, we replace every entry of its permutation matrix by a $2\times 2$ matrix according to the following rules:
Every $0$ entry is mapped onto a $2 \times 2$ matrix of type $A$, while 
every $1$ entry is mapped onto ${\footnotesize \left[\begin{array}{cc}
          1 & 1\\
          1 & 0 
         \end{array}\right]}$. 
This mapping (illustrated in Figure~\ref{fig:inv}) guarantees that the set of cells obtained is connected (hence is a polyomino), and avoids the submatrices $H'$ and $V'$, concluding the proof.
\end{proof}

\begin{figure}[htd]
\begin{center}
\includegraphics[width=8cm]{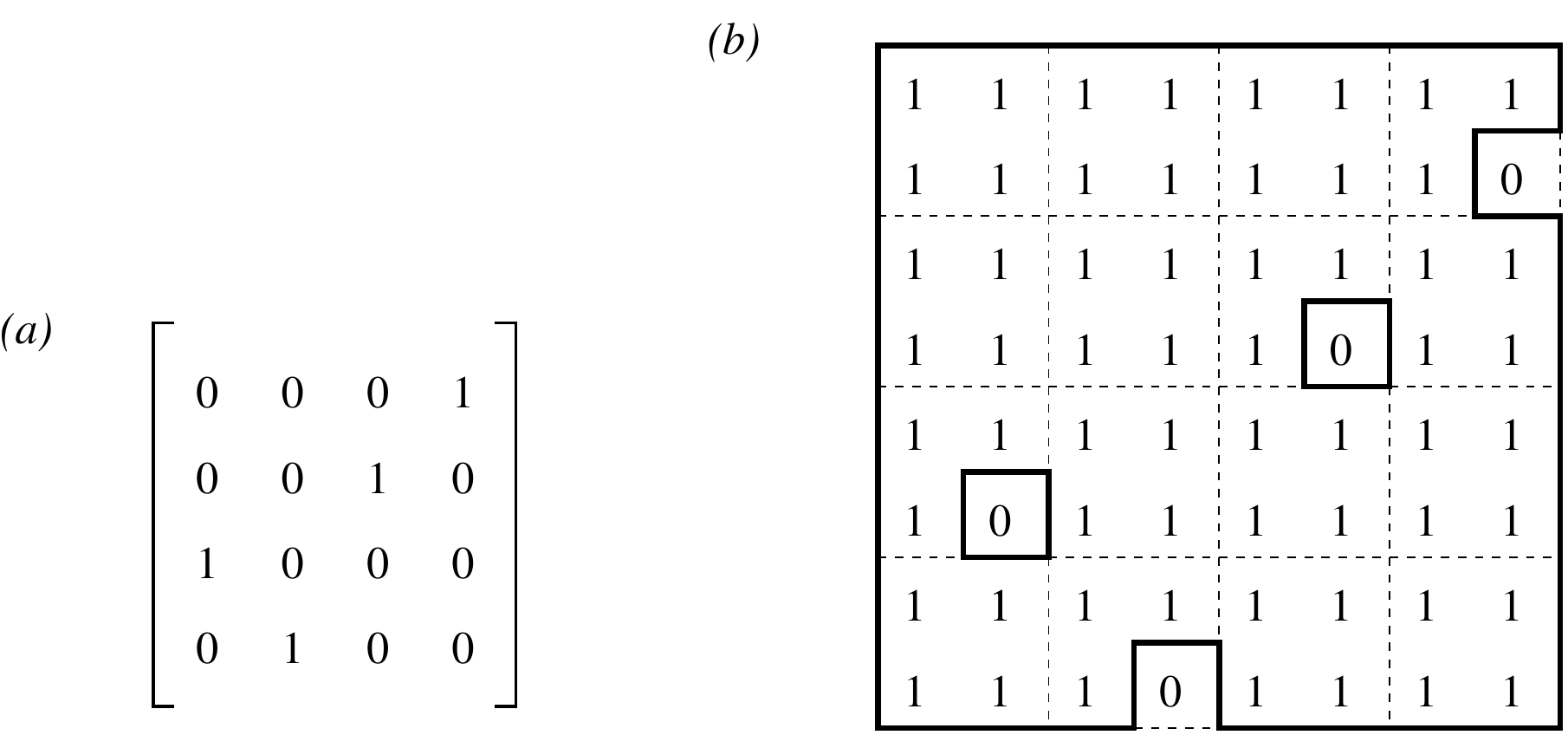}
\caption{$(a)$ a permutation matrix of dimension $4$; $(b)$ the corresponding polyomino of dimension $8$ in ${\cal C}'$.}
\label{fig:inv}
\end{center}
\end{figure}

\section{Some directions for future research}
\label{sec:further_research}

Because it introduces the new approach of submatrix avoidance in the study of permutation and polyomino classes, 
our work opens numerous and various directions for future research. 

\smallskip

Both in the case of permutation and polyomino classes, 
we have described several notions of bases for these classes: $p$-basis, $m$-bases, canonical $m$-basis, minimal $m$-bases. 
Section~\ref{sec:from_one_basis_to_another} explains how to describe the $p$-basis from any $m$-basis. 
Conversely, we may ask how to transform the $p$-basis into an ``efficient'' $m$-basis. Of course, the $p$-basis is itself an $m$-basis, but we may wish to describe the canonical one, or a minimal one. 

Many questions may also be asked about the canonical and minimal $m$-bases themselves. 
For instance: 
When does a class have a unique minimal $m$-basis? 
Which elements of the canonical $m$-basis may belong to a minimal $m$-basis? 
May we describe (or compute) the minimal $m$-bases from the canonical $m$-basis? 

Finally, we can study the classes for which the $p$-basis is itself a minimal $m$-basis of the class (see the examples of the polyomino classes of vertical bars, or of parallelogram polyominoes). 

\smallskip

Submatrix avoidance in permutation classes has allowed us to derive a statement (Corollary~\ref{cor:WE}) from which infinitely many Wilf-equivalences follow. 
Such general results on Wilf-equivalences are rare in the permutation patterns literature, and it would be interesting to explore how much further we can go in the study of Wilf-equivalences with the submatrix avoidance approach.

\smallskip

The most original concept of this article is certainly the introduction of the polyomino classes, which opens many directions for future research. 

One is a systematic study of polyomino classes defined by pattern avoidance. 
Because enumeration is the biggest open question about polyominoes, we should study the enumeration of such classes, and see whether some interesting bounds can be provided. Notice that the Stanley-Wilf-Marcus-Tardos theorem~\cite{marcTard} on permutation classes implies that the permutations in any given class represent a negligible proportion of all permutations. We don't know if a similar statement holds for polyomino classes. 

As we have reported in Section~\ref{sec:papc}, the poset $({\poly}, \polypattern )$ of polyominoes was introduced in \cite{CR}, where the authors proved that it is a ranked poset, and contains infinite antichains. There are however some combinatorial and algebraic properties of this poset which are still to explore, in particular w.r.t. characterizing some simple intervals in this poset.  

\smallskip

Finally, we have used binary matrices to import some questions on permutation classes to the context of polyominoes. But a similar approach could be applied to any other family of combinatorial objects which are represented by binary matrices. 

\paragraph*{Acknowledgments.}
We would like to thank Valentin F\'eray, Samanta Socci and Laurent Vuillon for helpful discussions on the topics of this paper. 
Many thanks also to the anonymous referee for carefully reading our work, and for pointing out some inaccuracies that have now been fixed. 
Part of the work in this paper was supported by the software suite \emph{PermLab}~\cite{PermLab1.0}.

\end{document}